\documentclass[12pt]{amsart}

\usepackage{aliascnt,amsmath,amssymb,amsthm,enumerate,mathrsfs,mathtools,xcolor,bm}
\usepackage{tikz}

\usepackage[T1]{fontenc}

\usepackage[marginparwidth=0pt,margin=24truemm]{geometry}

\definecolor{mylinkcolor}{RGB}{16, 156, 81}
\definecolor{mycitecolor}{RGB}{20, 80, 140}

\usepackage{hyperref}
\usepackage[nameinlink]{cleveref}
\hypersetup{
    setpagesize=false,
    bookmarksnumbered=true,
    bookmarksopen=true,
    colorlinks=true,
    linkcolor=mylinkcolor,
    citecolor=mycitecolor,
}
\usepackage{autonum}
\numberwithin{equation}{section}

\theoremstyle{plain}

\newtheorem{thm}{Theorem}[section]
\crefname{thm}{Theorem}{Theorems}

\newaliascnt{lem}{thm}
\newtheorem{lem}[lem]{Lemma}
\aliascntresetthe{lem}
\crefname{lem}{Lemma}{Lemmas}

\newaliascnt{prop}{thm}
\newtheorem{prop}[prop]{Proposition}
\aliascntresetthe{prop}
\crefname{prop}{Proposition}{Propositions}

\newaliascnt{cor}{thm}
\newtheorem{cor}[cor]{Corollary}
\aliascntresetthe{cor}
\crefname{cor}{Corollary}{Corollaries}

\newaliascnt{conj}{thm}

\aliascntresetthe{conj}
\crefname{conj}{Conjecture}{Conjectures}

\newaliascnt{rem}{thm}
\newtheorem{rem}[rem]{Remark}
\aliascntresetthe{rem}
\crefname{rem}{Remark}{Remarks}

\newtheorem{mainthm}{Main Theorem}

\crefname{mainthm}{Main Theorem}{Main Theorems}

\allowdisplaybreaks[2]
\everymath{\displaystyle}

\newcommand{\summ}[1]{\sum_{\substack{#1}}}

\newcommand{\QQ}{\mathbb{Q}}
\newcommand{\ZZ}{\mathbb{Z}}
\newcommand{\RR}{\mathbb{R}}
\newcommand{\CC}{\mathbb{C}}
\newcommand{\HH}{\mathbb{H}}

\DeclareMathOperator{\Span}{span}
\DeclareMathOperator{\SL}{SL}
\DeclareMathOperator{\wt}{wt}

\newcommand{\sltwo}{\mathfrak{sl}_2}
\newcommand{\mes}{\mathcal{E}}
\newcommand{\mesadm}{\mathcal{E}^{\mathrm{adm}}}
\newcommand{\ha}{\mathfrak{H}}
\newcommand{\hadm}{\mathfrak{H}^{\mathrm{adm}}}
\newcommand{\ordN}{\succ_N}

\newcommand{\lengthfun}[1]{\mathcal{L}(#1)}

\newcommand{\qmf}{\widetilde{\mathcal{M}}}

\newcommand{\kk}{{\bm{k}}}
\newcommand{\LL}{\mathcal{L}}
\newcommand{\ind}{\mathbf{1}}
\newcommand{\latzero}{\Lambda_\tau\setminus\{0\}}
\newcommand{\equivdef}{\xLeftrightarrow{\mathrm{def}}}
\newcommand{\equiviff}{\xLeftrightarrow{\mathrm{iff}}}
\newcommand{\relmid}{\mathrel{}\middle|\mathrel{}}

\title[The $\mathfrak{sl}_2$-algebra structure of multiple Eisenstein series]
{The $\mathfrak{sl}_2$-algebra structure of multiple Eisenstein series}

\author{Henrik Bachmann}
\address{Graduate School of Mathematics, Nagoya University, Nagoya, Japan.}
\email{henrik.bachmann@math.nagoya-u.ac.jp}

\author{Hayato Kanno}
\address{Mathematical Institute, Tohoku University, Sendai, Japan.}
\email{hayato.kanno.q1@dc.tohoku.ac.jp}

\date{\today}


\begin{document}

\begin{abstract}
We prove that the algebra of multiple Eisenstein series is an $\mathfrak{sl}_2$-algebra. In particular, we show that it is graded by weight, which also holds after taking complex linear spans. Further, we identify it as a graded $\QQ$-algebra with the associated graded algebra of $q$-analogues of multiple zeta values with respect to the weight filtration. The main ingredient is an estimate that compares the usual defining sums of multiple Eisenstein series with sums over a lattice order depending on an integer $N$.
\end{abstract}

\maketitle

\section{Introduction}

In this paper, we show that the algebra of multiple Eisenstein series is an $\sltwo$-algebra, which was conjectured in \cite{BIM}. Multiple Eisenstein series were introduced by Gangl, Kaneko and Zagier in \cite{GKZ} (see also \cite{Ba1}) and they are holomorphic functions on the upper half-plane which connect multiple zeta values and modular forms.

For integers $r\geq1$, $k_1\geq2$ and $k_2,\dots,k_r\geq1$, the \emph{multiple zeta value} is defined by
\begin{align}
\zeta(k_1,\dots,k_r)
\coloneqq
\sum_{n_1>\cdots>n_r>0}
\frac{1}{n_1^{k_1}\cdots n_r^{k_r}}\,.
\end{align}
Its weight is $k_1+\cdots+k_r$ and if $\mathcal{Z}_k$ denotes the $\QQ$-span of all multiple zeta values of weight $k$, with $\mathcal{Z}_0=\QQ$, then it is conjectured that
\begin{align}
\mathcal{Z}\stackrel{?}{=}\bigoplus_{k\geq0}\mathcal{Z}_k\,,
\end{align}
where $\mathcal{Z}$ is the $\QQ$-algebra of multiple zeta values, i.e. conjecturally there are no linear relations among multiple zeta values of different weights. For multiple zeta values this seems to be out of reach with current methods.

For an integer $r\geq1$ and $k_1,\dots,k_r\geq2$, the \emph{multiple Eisenstein series} is defined by
\begin{align}\label{eq:def-mes}
G_{k_1,\dots,k_r}(\tau)
\coloneqq
\lim_{M\to\infty}\lim_{N\to\infty}
\summ{\lambda_1,\dots,\lambda_r\in\ZZ_M\tau+\ZZ_N\\
\lambda_1\succ\cdots\succ\lambda_r\succ0}
\frac{1}{\lambda_1^{k_1}\cdots\lambda_r^{k_r}}\,,
\end{align}
where $\tau\in\HH$, $\ZZ_M=\{m\in\ZZ\mid |m|<M\}$, and the order on $\ZZ\tau+\ZZ$ is given by
\begin{align}\label{eq:mes order}
m_1\tau+n_1\succ m_2\tau+n_2
\quad\Longleftrightarrow\quad
m_1>m_2\quad\text{or}\quad
(m_1=m_2\ \text{and}\ n_1>n_2)\,.
\end{align}
For $k_1\geq3$ the sum in \eqref{eq:def-mes} converges absolutely and the iterated limit is just needed for the case $k_1=2$. The constant term in the Fourier expansion of $G_{k_1,\dots,k_r}$ is given by the multiple zeta value $\zeta(k_1,\dots,k_r)$.

We denote the $\QQ$-linear space spanned by all multiple Eisenstein series by
\begin{align}
\mes\coloneqq
\Span_\QQ\{G_{k_1,\dots,k_r}(\tau),G_\varnothing(\tau)
\mid r\geq1,\ k_1,\dots,k_r\geq2\}\,,
\end{align}
where $G_\varnothing(\tau)=1$. Since multiple Eisenstein series satisfy the harmonic product formula, which follows directly from their definition, the space $\mes$ is a $\QQ$-algebra. For $k\geq1$, we put
\begin{align}
\mes_k\coloneqq
\Span_\QQ\{G_{k_1,\dots,k_r}(\tau)\mid
k_1,\dots,k_r\geq2,\ k_1+\cdots+k_r=k\},
\qquad
\mes_0\coloneqq\QQ\,.
\end{align}
We also consider the complex linear spans
\begin{align}\label{eq:def-complex-E}
\mes_k^\CC
&\coloneqq
\Span_\CC\{G_{k_1,\dots,k_r}(\tau)\mid
k_1,\dots,k_r\geq2,\ k_1+\cdots+k_r=k\},
\\
\mes^\CC
&\coloneqq
\Span_\CC(\mes)
=\sum_{k\geq0}\mes_k^\CC
\subset\mathcal{O}(\HH)\,,
\end{align}
where $\mes_0^\CC=\CC$. Notice that $\mes^\CC$ is defined as a space of holomorphic functions and therefore it is a priori not clear that the natural surjection $\mes\otimes_\QQ\CC\to\mes^\CC$ is injective.

Next, we recall the $q$-series which appear in the Fourier expansion of multiple Eisenstein series. For this, we put $q=e^{2\pi i\tau}$ and define for $k_1,\dots,k_r\geq1$
\begin{align}\label{eq:def-g}
g_{k_1,\dots,k_r}(\tau)
\coloneqq
\frac{1}{(k_1-1)!\cdots(k_r-1)!}
\summ{c_1>\cdots>c_r>0\\d_1,\dots,d_r>0}
d_1^{k_1-1}\cdots d_r^{k_r-1}
q^{c_1d_1+\cdots+c_rd_r}\,.
\end{align}
These $q$-series are $q$-analogues of multiple zeta values, which were studied by the first author and K\"uhn in \cite{BK1}. Writing $k=k_1+\cdots+k_r$, the top-weight term in the Fourier expansion of $G_{k_1,\dots,k_r}$ is given by $(-2\pi i)^k g_{k_1,\dots,k_r}$. Although the $q$-series $g_{k_1,\dots,k_r}$ are defined for arbitrary $k_1,\dots,k_r\geq1$, we will just consider the case $k_1,\dots,k_r\geq2$ and put
\begin{align}
\mathsf{qMZV}
=\QQ+
\Span_\QQ\{g_{k_1,\dots,k_r}\mid
r\geq1,\ k_1,\dots,k_r\geq2\}\,.
\end{align}
This space is a realization of the algebra considered by Okounkov in \cite{O}, as shown in \cite[Section~2]{BK2}. For $k\geq0$, we denote by $\operatorname{Fil}^W_k\mathsf{qMZV}$ the $\QQ$-span of $1$ and all $g_{k_1,\dots,k_r}$ with $k_1,\dots,k_r\geq2$ and $k_1+\cdots+k_r\leq k$. The product of two of these $q$-series can again be written as a linear combination of them (see \cite{BK1}), which makes $\mathsf{qMZV}$ a $\QQ$-algebra and gives the multiplicative filtration
\begin{align}\label{eq:qMZV-filtered-product}
\operatorname{Fil}^W_k\mathsf{qMZV}\cdot
\operatorname{Fil}^W_\ell\mathsf{qMZV}
\subset
\operatorname{Fil}^W_{k+\ell}\mathsf{qMZV}\,.
\end{align}
We set $\operatorname{Fil}^W_{-1}\mathsf{qMZV}=0$ and put
\begin{align}
\operatorname{gr}^W_k\mathsf{qMZV}
&\coloneqq
\operatorname{Fil}^W_k\mathsf{qMZV}/
\operatorname{Fil}^W_{k-1}\mathsf{qMZV},
\\
\operatorname{gr}^W\mathsf{qMZV}
&\coloneqq
\bigoplus_{k\geq0}\operatorname{gr}^W_k\mathsf{qMZV}\,.
\end{align}

Similar to the case of multiple zeta values, one expects that there are no linear relations among multiple Eisenstein series of different weights. This will follow from the main result of this paper, which states that $\mes$ is an \emph{$\sltwo$-algebra}, i.e. an algebra equipped with three derivations $W,D,\delta$ satisfying
\begin{align}\label{eq:sl2-relations}
[W,D]=2D,\qquad [W,\delta]=-2\delta,\qquad [\delta,D]=W\,.
\end{align}
The classical example of an $\sltwo$-algebra is the $\QQ$-algebra of quasimodular forms for $\SL_2(\ZZ)$, which is given by
\begin{align}
\qmf^\QQ\coloneqq\QQ[G_2,G_4,G_6]\,.
\end{align}
Here the three derivations are given by
\begin{align}\label{eq:qmf-triple}
W(G_k)&=kG_k\qquad (k=2,4,6),
&
D(f)&=2\pi i\frac{d}{d\tau}f,
\\
\delta(G_2)&=-\frac12,
&
\delta(G_4)&=\delta(G_6)=0\,.
\end{align}
By the Ramanujan differential equations, the operator $D$ preserves $\qmf^\QQ$ and these derivations satisfy \eqref{eq:sl2-relations} (see \cite{Za}).

Since the Eisenstein series $G_2,G_4$ and $G_6$ are multiple Eisenstein series of depth one and $\mes$ is an algebra, we have $\qmf^\QQ\subset\mes$. So it is natural to ask if the $\sltwo$-algebra structure of $\qmf^\QQ$ extends to $\mes$. Such an extension was first constructed in a formal setting by the first author and van Ittersum in \cite{BIM}, where the algebra $\mathcal{G}^{\mathrm f}$ of formal multiple Eisenstein series was introduced and shown to be an $\sltwo$-algebra. The subalgebra
\begin{align}
\mes^{\mathrm f}
\coloneqq
\QQ\oplus
\Span_\QQ\{G^{\mathrm f}(k_1,\dots,k_r)\mid
r\geq1,\ k_1,\dots,k_r\geq2\}
\end{align}
is the formal analogue of $\mes$. In \cite{BIM} it was conjectured that $\mes^{\mathrm f}$ is an $\sltwo$-subalgebra of $\mathcal{G}^{\mathrm f}$, which was then proven in \cite[Main Theorem~D]{BKM}. The bi-multiple Eisenstein realization in \cite[Main Theorem~C]{BKM} also shows that the assignment
\begin{align}
\label{eq:formal-realization}
\rho:\mes^{\mathrm f}&\longrightarrow\mes,\\
G^{\mathrm f}(k_1,\dots,k_r)&\longmapsto G_{k_1,\dots,k_r}
\end{align}
gives a well-defined surjective algebra homomorphism. This map is expected to be an isomorphism, but its injectivity is not known and will not be discussed in this paper.

In \cite{BIM} it was also conjectured that the formal $\sltwo$-action induces an $\sltwo$-action on $\mes$, and the proof of this conjecture is the main result of this paper.

\begin{mainthm}\label{thm:main-sl2}
\begin{enumerate}[(i)]
\item The assignments
\begin{align}\label{eq:analytic-triple}
WG_{k_1,\dots,k_r}
&=(k_1+\cdots+k_r)G_{k_1,\dots,k_r},\\
DG_{k_1,\dots,k_r}
&=2\pi i\frac{d}{d\tau}G_{k_1,\dots,k_r},\\
\delta G_{k_1,\dots,k_r}
&=
\begin{cases}
-\dfrac12G_{k_2,\dots,k_r},&k_1=2,\\
0,&k_1>2,
\end{cases}
\end{align}
on the generators of $\mes$ extend to well-defined derivations $W$, $D$ and $\delta$ on $\mes$.

\item The derivations $W$, $D$ and $\delta$ satisfy
\begin{align}
[W,D]=2D,\qquad [W,\delta]=-2\delta,\qquad [\delta,D]=W\,.
\end{align}
In particular, $\mes$ is an $\sltwo$-algebra.

\item Put
\begin{align}
\mesadm
&\coloneqq
\QQ\oplus
\Span_\QQ\{G_{k_1,\dots,k_r}\mid
r\geq1,\ k_1\geq3,\ k_2,\dots,k_r\geq2\}\,.
\end{align}
Then
\begin{align}
\mes=\mesadm[G_2]
\end{align}
is a free polynomial algebra over $\mesadm$, and
\begin{align}
\delta=-\frac12\frac{\partial}{\partial G_2}\,.
\end{align}
\end{enumerate}
\end{mainthm}

Notice that the well-definedness of $W$ in (i) implies that
\begin{align}
\mes=\bigoplus_{k\geq0}\mes_k\,,
\end{align}
i.e. there are no linear relations among multiple Eisenstein series of different weights. That $D$ is well-defined was already shown in \cite{BKM}. Our proof of the weight grading actually gives the following stronger statement for the complex linear span, and together with the comparison of Fourier expansions in \cite{BT} and the rational-structure argument in \cite{HST} we also obtain an identification of the associated graded algebra of $\mathsf{qMZV}$. Since this identifies $\mes$ with $\operatorname{gr}^W\mathsf{qMZV}$, the conjectural dimensions of $\operatorname{gr}^W_k\mathsf{qMZV}$ given by Okounkov in \cite{O} and by the first author and K\"uhn in \cite{BK2} are also the conjectural dimensions of the spaces $\mes_k$, which explains the dimension conjecture for multiple Eisenstein series discussed in \cite{BKM}.

\begin{mainthm}\label{thm:main-complex}
\begin{enumerate}[(i)]
\item The $\CC$-span of multiple Eisenstein series is graded by weight:
\begin{align}
\mes^\CC=\bigoplus_{k\geq0}\mes_k^\CC\,.
\end{align}
In particular, the $\QQ$-algebra of multiple Eisenstein series is graded by weight:
\begin{align}
\mes=\bigoplus_{k\geq0}\mes_k\,.
\end{align}
\item The assignment
\begin{align}\label{eq:Phi-qMZV}
\Phi_k:\mes_k&\longrightarrow
\operatorname{gr}^W_k\mathsf{qMZV},\\
G_{k_1,\dots,k_r}
&\longmapsto
g_{k_1,\dots,k_r}
\bmod\operatorname{Fil}^W_{k-1}\mathsf{qMZV}
\end{align}
is a well-defined isomorphism of $\QQ$-vector spaces, where $\Phi_0(1)=1$.  These maps give an isomorphism of graded $\QQ$-algebras
\begin{align}
\Phi:\mes\xrightarrow{\,\sim\,}
\operatorname{gr}^W\mathsf{qMZV}\,.
\end{align}
Moreover, the map
\begin{align}
\mes\otimes_\QQ\CC\xrightarrow{\,\sim\,}\mes^\CC
\end{align}
is an isomorphism.
\end{enumerate}
\end{mainthm}

The structure of this paper is as follows. In \Cref{sec:regularization}, we introduce an algebraic setup for multiple Eisenstein series and show that the well-definedness of $\delta$ is equivalent to $G_2$ being transcendental over the algebra of absolutely convergent multiple Eisenstein series. The main analytic ingredient is an estimate for sums over the moving lattice orders, which we introduce in \Cref{sec:moving-orders} and whose proof is given in \Cref{sec:proof-first-order}. This estimate is then used in \Cref{sec:weight} to prove \cref{thm:main-complex} and in \Cref{sec:delta} to complete the proof of \cref{thm:main-sl2}.

\vspace{0.5cm}

\noindent\textbf{Acknowledgements.} This project was partially supported by JSPS KAKENHI Grant Number JP26K22254 (H.B.) and 26KJ0525 (H.K).

\section{Polynomial regularization with respect to \texorpdfstring{$z_2$}{z2}}
\label{sec:regularization}

In this section, we want to describe the algebraic structure behind the lowering operator $\delta$. For this, we first define the space
\begin{align}
\ha^{\geq2}\coloneqq\QQ\langle z_k\mid k\geq2\rangle\,,
\end{align}
which is the noncommutative polynomial ring in the letters $z_k$ with $k\geq 2$, and we equip it with the harmonic product, which is defined recursively by
\begin{align}
w\ast1=1\ast w=w
\end{align}
and
\begin{align}\label{eq:harmonic-product}
z_ku\ast z_lv
=z_k(u\ast z_lv)+z_l(z_ku\ast v)+z_{k+l}(u\ast v)\,.
\end{align}
It gives $\ha^{\geq2}$ the structure of a commutative $\QQ$-algebra, which we denote by $\ha^{\geq2}_\ast$.

Notice that the $\QQ$-linear map
\begin{align}\label{eq:realization-G}
G:\ha^{\geq2}_\ast&\longrightarrow\mes,\\
z_{k_1}\cdots z_{k_r}&\longmapsto G_{k_1,\dots,k_r}
\end{align}
is a surjective algebra homomorphism, which also includes the words starting with the letter $z_2$.

The weight of a word is given by
\begin{align}
\wt(1)=0,\qquad
\wt(z_{k_1}\cdots z_{k_r})=k_1+\cdots+k_r\,.
\end{align}
We denote by $\ha^{\geq2}_k$ the span of all words of weight $k$. The harmonic product is compatible with the weight and the map in \eqref{eq:realization-G} sends a word of weight $k$ to $\mes_k$, which does not use the direct sum decomposition in \cref{thm:main-complex}.

We also define the subspace of admissible words, i.e. those for which $G$ converges absolutely, by
\begin{align}\label{eq:def-Hadm}
\hadm
=
\QQ\oplus
\Span_\QQ\{z_{k_1}\cdots z_{k_r}\mid r\geq1,\ k_1\geq3,\ k_2,\dots,k_r\geq2\}\,.
\end{align}
We write
\begin{align}
\hadm_k\coloneqq\hadm\cap\ha^{\geq2}_k
\end{align}
for its homogeneous part of weight $k$. By \eqref{eq:harmonic-product} the space $\hadm$ is a subalgebra of $\ha^{\geq2}_\ast$, which we denote by $\hadm_\ast$ when we want to emphasize the harmonic product.

\begin{prop}\label{prop:polynomial-regularization}
There is an isomorphism of weight-graded algebras
\begin{align}\label{eq:H-polynomial}
\ha^{\geq2}_\ast\cong\hadm_\ast[z_2]\,.
\end{align}
Equivalently, every $w\in\ha^{\geq2}$ can be written uniquely as
\begin{align}\label{eq:canonical-decomposition}
w=\sum_{j=0}^d a_j\ast z_2^{\ast j},
\qquad a_j\in\hadm\,,
\end{align}
where $z_2^{\ast0}=1$.
\end{prop}

\begin{proof}
This is the usual harmonic regularization. Every word can be written uniquely as $z_2^mv$, where $v=1$ or $v$ starts with a letter $z_k$ with $k\geq3$, and by induction using \eqref{eq:harmonic-product} we get
\begin{align}
z_2^{\ast m}\ast v
=m!\,z_2^mv+
\text{a linear combination of words with fewer than $m$ initial
letters $z_2$}\,.
\end{align}
Since a word of weight $k$ corresponds to a composition of $k$ with all parts at least $2$, there are only finitely many words of a fixed weight, and the bases $\{z_2^mv\}$ and $\{v\ast z_2^{\ast m}\}$ are therefore related by an invertible triangular matrix. This proves the existence and uniqueness in \eqref{eq:canonical-decomposition}.
\end{proof}

Define the linear map $\bar\delta:\ha^{\geq2}\to\ha^{\geq2}$ by
\begin{align}\label{eq:formal-delta}
\bar\delta(z_{k_1}\cdots z_{k_r})
=
\begin{cases}
-\dfrac12z_{k_2}\cdots z_{k_r},&k_1=2,\\
0,&k_1>2,
\end{cases}
\qquad
\bar\delta(1)=0\,.
\end{align}
This is the restriction of the lowering operator on formal multiple Eisenstein series from \cite{BIM}.

\begin{lem}\label{lem:delta-polynomial}
The map $\bar\delta$ is a derivation of $\ha^{\geq2}_\ast$.  Under the isomorphism \eqref{eq:H-polynomial}, it is given by
\begin{align}\label{eq:delta-partial}
\bar\delta=-\frac12\frac{\partial}{\partial z_2}\,.
\end{align}
In particular, if $w$ has the decomposition \eqref{eq:canonical-decomposition}, then
\begin{align}\label{eq:delta-decomposition}
\bar\delta(w)
=-\frac12\sum_{j=1}^d
j\,a_j\ast z_2^{\ast(j-1)}\,.
\end{align}
\end{lem}

\begin{proof}
The derivation property follows directly from \eqref{eq:harmonic-product}, since a term starting with $z_2$ in the harmonic product of two words is obtained by taking the first letter $z_2$ of one of the two words without merging it, and removing this letter gives exactly
\begin{align}
\bar\delta(u\ast v)=\bar\delta(u)\ast v+u\ast\bar\delta(v)\,.
\end{align}
Since $\bar\delta$ vanishes on $\hadm$ and satisfies $\bar\delta(z_2)=-\frac12$, the uniqueness in \cref{prop:polynomial-regularization} gives \eqref{eq:delta-partial} and \eqref{eq:delta-decomposition}.
\end{proof}

It is not clear that the decomposition \eqref{eq:H-polynomial} is still valid after applying the map $G$, and in the following proposition we show that this is equivalent to $\bar\delta$ inducing a well-defined derivation on $\mes$. For this, we define the space of absolutely convergent multiple Eisenstein series, i.e. those with $k_1\geq 3$, by
\begin{align}
\mesadm=G(\hadm)\,.
\end{align}
Since $\hadm$ is a subalgebra, $\mesadm$ is a subalgebra of $\mes$, and by \cref{prop:polynomial-regularization} we have
\begin{align}\label{eq:E-generation}
\mes=\mesadm[G_2]\,,
\end{align}
but it is a priori not clear that this representation as a polynomial in $G_2$ is unique.

\begin{prop}\label{prop:delta-equivalences}
The following statements are equivalent.
\begin{enumerate}
\item The map $\bar\delta$ preserves $\ker(G)$ and therefore induces a well-defined derivation on $\mes$, again denoted by $\delta$.
\item The evaluation map
\begin{align}
\operatorname{ev}_{G_2}:\mesadm[T]&\longrightarrow\mes,\\
T&\longmapsto G_2
\end{align}
is injective.
\item The Eisenstein series $G_2$ is transcendental over $\mesadm$.
\end{enumerate}
If these statements hold, then
\begin{align}
\mes=\mesadm[G_2]
\end{align}
freely and
\begin{align}\label{eq:delta-G2}
\delta=-\frac12\frac{\partial}{\partial G_2}\,.
\end{align}
\end{prop}

\begin{proof}
Assume (i). Then the induced derivation $\delta$ vanishes on $\mesadm$ and satisfies $\delta(G_2)=-\frac12$. If we have
\begin{align}
b_0+b_1G_2+\cdots+b_dG_2^d=0,
\qquad b_j\in\mesadm\,.
\end{align}
then applying $\delta^d$ gives $d!\left(-\frac12\right)^db_d=0$,
and therefore $b_d=0$. Repeating this argument gives $b_j=0$ for all $j$, which proves (ii).

The equivalence of (ii) and (iii) is just the definition of transcendence. Now assume (ii). Then by \eqref{eq:E-generation} the evaluation map gives an isomorphism $\mesadm[T]\cong\mes$, and we can transport the derivation $-\frac12\partial/\partial T$ to $\mes$. If $w$ has the decomposition \eqref{eq:canonical-decomposition}, then \eqref{eq:delta-decomposition} shows that this derivation sends $G(w)$ to $G(\bar\delta w)$. So $\bar\delta$ preserves $\ker(G)$, which proves (i), and \eqref{eq:delta-G2} follows.
\end{proof}

So the remaining task is to prove $\bar\delta(\ker(G))\subset\ker(G)$, which we will deduce from the main estimate in \Cref{sec:moving-orders}.

\section{Moving lattice orders and the main estimate}\label{sec:moving-orders}

In this section, we want to introduce a family of orders on the lattice $\ZZ\tau+\ZZ$ depending on an integer $N\geq1$ and state the main estimate of this paper. For this, we fix $\tau\in\HH$ and write
\begin{align}
\Lambda_\tau=\ZZ\tau+\ZZ\,.
\end{align}
For an integer $N\geq1$ and $\lambda=a\tau+b\in\Lambda_\tau$, we define
\begin{align}
H_N(\lambda)=Na+b\,.
\end{align}
Using this, we define the positive cone
\begin{align}\label{eq:moving-cone}
P_N
\coloneqq
\{\lambda\in\Lambda_\tau\mid H_N(\lambda)>0\}
\cup
\{a\tau+b\in\Lambda_\tau\mid H_N(a\tau+b)=0,\ a<0\}\,.
\end{align}

For $\tau=i$, the following diagrams show the positive lattice points in the range $-4\leq a,b\leq4$.

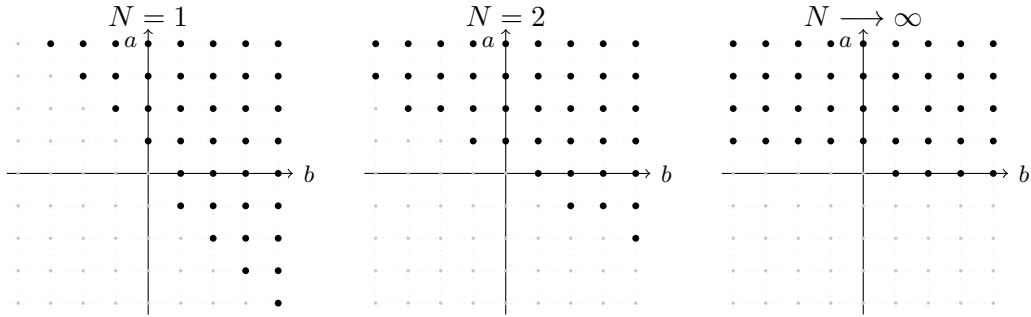
\begingroup
\newcommand{\finiteconepanel}[2]{%
  \begin{scope}[shift={(#1,0)}]
    \draw[step=1,gray!35,very thin,dotted]
      (-4.15,-4.15) grid (4.15,4.15);
    \draw[->,thin] (-4.35,0) -- (4.45,0)
      node[right,font=\scriptsize] {$b$};
    \draw[->,thin] (0,-4.35) -- (0,4.45);
    \node[left,font=\scriptsize] at (0,4.05) {$a$};
    \foreach \a in {-4,...,4}{
      \foreach \b in {-4,...,4}{
        \fill[gray!45] (\b,\a) circle (1.5pt);
        \pgfmathtruncatemacro{\hvalue}{#2*\a+\b}
        \ifnum\hvalue>0
          \fill[black] (\b,\a) circle (3pt);
        \else
          \ifnum\hvalue=0
            \ifnum\a<0
              \fill[black] (\b,\a) circle (3pt);
            \fi
          \fi
        \fi
      }
    }
    \node[font=\small] at (0,4.85) {$N=#2$};
  \end{scope}%
} \newcommand{\limitconepanel}[1]{%
  \begin{scope}[shift={(#1,0)}]
    \draw[step=1,gray!35,very thin,dotted]
      (-4.15,-4.15) grid (4.15,4.15);
    \draw[->,thin] (-4.35,0) -- (4.45,0)
      node[right,font=\scriptsize] {$b$};
    \draw[->,thin] (0,-4.35) -- (0,4.45);
    \node[left,font=\scriptsize] at (0,4.05) {$a$};
    \foreach \a in {-4,...,4}{
      \foreach \b in {-4,...,4}{
        \fill[gray!45] (\b,\a) circle (1.5pt);
        \ifnum\a>0
          \fill[black] (\b,\a) circle (3pt);
        \else
          \ifnum\a=0
            \ifnum\b>0
              \fill[black] (\b,\a) circle (3pt);
            \fi
          \fi
        \fi
      }
    }
    \node[font=\small] at (0,4.85) {$N\longrightarrow\infty$};
  \end{scope}%
}

\begin{center}
  \begin{tikzpicture}[scale=0.43]
    \finiteconepanel{0}{1}
    \finiteconepanel{11}{2}
    \limitconepanel{22}
  \end{tikzpicture}
\end{center}
\endgroup

The corresponding total order $\ordN$ is given by
\begin{align}\label{eq:moving-order}
a_1\tau+b_1\ordN a_2\tau+b_2
\end{align}
if $H_N(a_1\tau+b_1)>H_N(a_2\tau+b_2)$, or if these values are equal and $a_1<a_2$, and we therefore have
\begin{align}
P_N=\{\lambda\mid\lambda\ordN0\}\,.
\end{align}

Notice that for two fixed distinct lattice points, the order $\ordN$ agrees with the order $\succ$ for $N$ large enough, since the sign of
\begin{align}
H_N((a_1-a_2)\tau+(b_1-b_2))
=N(a_1-a_2)+(b_1-b_2)
\end{align}
is determined by $a_1-a_2$ for large $N$, and by $b_1-b_2$ if $a_1=a_2$. If $a_1\neq a_2$, then this expression vanishes for at most one $N$, and for this $N$ the tie-break $a_1<a_2$ in \eqref{eq:moving-order} is used, which does not matter for large $N$.

For $k_1\geq3$ and $k_2,\dots,k_r\geq2$, we define the absolutely convergent sum
\begin{align}\label{eq:moving-G}
G^{(N)}_{k_1,\dots,k_r}(\tau)
\coloneqq
\summ{\lambda_1,\dots,\lambda_r\in\Lambda_\tau\\\lambda_1\ordN\cdots\ordN\lambda_r\ordN0}
\frac{1}{\lambda_1^{k_1}\cdots\lambda_r^{k_r}}\,.
\end{align}

For later use, we put
\begin{align}
z_N=\tau-N,
\qquad
\tau_N=-\frac{1}{z_N}\,.
\end{align}

\begin{lem}\label{lem:coordinate-change}
Let $k=k_1+\cdots+k_r$, where $k_1\geq3$ and $k_2,\dots,k_r\geq2$.  Then
\begin{align}\label{eq:exact-moving-identity}
\tau_N^{k}G_{k_1,\dots,k_r}(\tau_N)
=G^{(N)}_{k_1,\dots,k_r}(\tau)\,.
\end{align}
\end{lem}

\begin{proof}
For a lattice point $m\tau_N+n$, put
\begin{align}
\lambda
\coloneqq
m-nz_N
=(-n)\tau+(m+Nn)\,.
\end{align}
The coordinate map
\begin{align}
(m,n)\longmapsto(a,b)=(-n,m+Nn)
\end{align}
is given by a matrix in $\SL_2(\ZZ)$ and therefore a bijection. Writing $\lambda=a\tau+b$, we have
\begin{align}
a=-n,\qquad b=m+Nn,\qquad H_N(\lambda)=m\,,
\end{align}
and
\begin{align}
m\tau_N+n=-\frac{\lambda}{z_N}\,.
\end{align}
The condition $m>0$, or $m=0$ and $n>0$, becomes $H_N(\lambda)>0$, or $H_N(\lambda)=0$ and $a<0$, and the lexicographic order given first by $m$ and then by $n$ becomes the order $\ordN$ in \eqref{eq:moving-order}. This change of coordinates therefore gives the factor $(-1)^kz_N^k$ and the claimed identity. Notice that the assumption $k_1\geq3$ is needed here, since it makes both sides absolutely convergent, which justifies the reordering of the sum. For $k_1=2$ one would also need to transport the Eisenstein summation in \eqref{eq:def-mes}.
\end{proof}

\begin{rem}\label{rem:sl2-moving-cone}
The cones $P_N$ can also be obtained from the $\SL_2(\ZZ)$-action on positive subsets of $\ZZ^2$ considered in \cite[Sections~4.1 \& ~4.2]{Ba1}. For this, we identify $a\tau+b$ with the row vector $(a,b)$ and put
\begin{align}
P_{\mathrm{std}}
\coloneqq
\{(a,b)\in\ZZ^2\mid a>0\text{ or }(a=0\text{ and }b>0)\}\,.
\end{align}
For $\gamma\in\SL_2(\ZZ)$, we consider the action $\gamma\mathbin{\cdot}P=\{v\in\ZZ^2\mid v\gamma\in P\}$. With
\begin{align}
T=
\begin{pmatrix}1&1\\0&1\end{pmatrix},
\qquad
S=
\begin{pmatrix}0&-1\\1&0\end{pmatrix},
\qquad
\gamma_N=T^NS=
\begin{pmatrix}N&-1\\1&0\end{pmatrix}\,,
\end{align}
we have
\begin{align}
(a,b)\gamma_N=(Na+b,-a)\,.
\end{align}
This shows that the elements of $P_N$ correspond exactly to $\gamma_N\mathbin{\cdot}P_{\mathrm{std}}$. We also have
\begin{align}
\gamma_N^{-1}
=
\begin{pmatrix}0&1\\-1&N\end{pmatrix},
\qquad
(m,n)\gamma_N^{-1}=(-n,m+Nn)\,,
\end{align}
which is exactly the change of coordinates in the proof of \cref{lem:coordinate-change}, and the M\"obius action of $\gamma_N^{-1}$ gives
\begin{align}
\gamma_N^{-1}\tau=\frac{1}{N-\tau}
=-\frac{1}{\tau-N}\,,
\end{align}
which is the point $\tau_N$ used there. The conjugated shear $S^{-1}T^NS=\left(\begin{smallmatrix}1&0\\-N&1\end{smallmatrix}\right)$ also appears here, and for $N=1$ it represents the same element of $\operatorname{PSL}_2(\ZZ)$ as $STS$.
\end{rem}

For $t\geq0$, we write
\begin{align}
\lengthfun{t}=1+\log(2+t)\,.
\end{align}

For a homogeneous $w\in\ha^{\geq2}$ of weight $k$, we define
\begin{align}\label{eq:def-SN}
G^{(N)}(w)(\tau)
\coloneqq
\tau_N^{k}
G(w)(\tau_N)\,.
\end{align}
For fixed $N$, this map is multiplicative on homogeneous elements, i.e. we have
\begin{align}\label{eq:SN-multiplicative}
G^{(N)}(u\ast v)
=G^{(N)}(u)G^{(N)}(v)\,.
\end{align}

The following estimate is the main analytic result of this paper, and its proof will be given in \Cref{sec:proof-first-order}.

\begin{thm}\label{thm:first-order}
For every index $k_1\geq3$, $k_2,\dots,k_r\geq2$, there is an exponent $M\geq0$ such that, locally uniformly for $\tau\in\HH$,
\begin{align}\label{eq:first-order-index}
G^{(N)}_{k_1,\dots,k_r}(\tau)-G_{k_1,\dots,k_r}(\tau)
=
O\left(\frac{\lengthfun{N}^M}{N^2}\right)\,.
\end{align}
\end{thm}

We first give some direct consequences of \cref{thm:first-order}.

\begin{cor}\label{cor:absolute-limit}
Let $k_1\geq3$ and $k_2,\dots,k_r\geq2$, and put $k=k_1+\cdots+k_r$.  Then, locally uniformly for $\tau\in\HH$,
\begin{align}\label{eq:absolute-limit}
\tau_N^k
G_{k_1,\dots,k_r}\left(\tau_N\right)
\longrightarrow
G_{k_1,\dots,k_r}(\tau)\,.
\end{align}
\end{cor}

\begin{proof}
By \eqref{eq:exact-moving-identity}, the left-hand side of \eqref{eq:absolute-limit} equals $G^{(N)}_{k_1,\dots,k_r}(\tau)$, and the statement follows from \cref{thm:first-order}.
\end{proof}

\begin{cor}\label{cor:Hadm-limit}
If $a\in\hadm$ is homogeneous, then for some exponent $M$,
\begin{align}\label{eq:first-order}
G^{(N)}(a)-G(a)
=
O\left(\frac{\lengthfun{N}^M}{N^2}\right)
=o(N^{-1})
\end{align}
locally uniformly for $\tau\in\HH$.
\end{cor}

\begin{proof}
For a nonempty word this follows from \eqref{eq:exact-moving-identity} and \cref{thm:first-order}, for the empty word it is clear, and the general case follows by linearity.
\end{proof}

\section{Consequences: the weight grading and
\texorpdfstring{$\mathsf{qMZV}$}{qMZV}}
\label{sec:weight}

In this section, we want to prove \cref{thm:main-complex}. By \cref{cor:Hadm-limit}, the first-order term in $G^{(N)}(w)$ vanishes for $w\in\hadm$, and we will see that the only first-order contribution comes from $G_2$. Writing $E_2=1-24\sum_{n\geq1}\sigma_1(n)q^n$, the Eisenstein summation in \eqref{eq:def-mes} gives $G_2=\frac{\pi^2}{6}E_2$, and the classical transformation law for $E_2$ (see \cite{Za}) becomes
\begin{align}\label{eq:G2-transformation}
G_2\left(-\frac1z\right)=z^2G_2(z)-\pi iz\,.
\end{align}
Since $G_2(\tau-N)=G_2(\tau)$, we obtain
\begin{align}\label{eq:SN-z2}
G^{(N)}(z_2)(\tau)
=G_2(\tau)-\frac{\pi i}{z_N}\,.
\end{align}

\begin{prop}\label{prop:two-term}
If $w\in\ha^{\geq2}$ is homogeneous of weight $k$, then there is an exponent $M$ such that
\begin{align}\label{eq:two-term}
G^{(N)}(w)(\tau)
=G(w)(\tau)
+\frac{2\pi i}{z_N}G(\bar\delta w)(\tau)
+O\left(\frac{\lengthfun{N}^M}{N^2}\right)
\end{align}
locally uniformly for $\tau\in\HH$.
\end{prop}

\begin{proof}
By \cref{prop:polynomial-regularization} we can write
\begin{align}
w=\sum_{j=0}^d a_j\ast z_2^{\ast j},
\qquad
a_j\in\hadm_{k-2j}\,,
\end{align}
and by \eqref{eq:SN-multiplicative} and \eqref{eq:SN-z2} we get
\begin{align}\label{eq:SN-polynomial}
G^{(N)}(w)
=\sum_{j=0}^d
G^{(N)}(a_j)
\left(G_2-\frac{\pi i}{z_N}\right)^j\,.
\end{align}
Using \cref{cor:Hadm-limit} in \eqref{eq:SN-polynomial}, we obtain
\begin{align}
G^{(N)}(w)
&=
\sum_{j=0}^d
\left(G(a_j)+O(N^{-2}\lengthfun{N}^M)\right)
\left(G_2-\frac{\pi i}{z_N}\right)^j\\
&=
G(w)
-\frac{\pi i}{z_N}
\sum_{j=1}^d jG(a_j)G_2^{j-1}
+O(N^{-2}\lengthfun{N}^{M'})\,.
\end{align}
Since by \eqref{eq:delta-decomposition} we have
\begin{align}
G(\bar\delta w)
=-\frac12\sum_{j=1}^d jG(a_j)G_2^{j-1}\,,
\end{align}
this gives \eqref{eq:two-term}.
\end{proof}

\begin{cor}\label{cor:full-limit}
For every homogeneous $w\in\ha^{\geq2}$, we have
\begin{align}\label{eq:full-limit}
G^{(N)}(w)(\tau)\longrightarrow G(w)(\tau)
\end{align}
locally uniformly for $\tau\in\HH$.
\end{cor}

\begin{proof}
This follows directly from \cref{prop:two-term}, since $|z_N|\asymp N$ locally uniformly and $\lengthfun{N}^M/N^2\to0$.
\end{proof}

\begin{proof}[Proof of \cref{thm:main-complex} \textup{(i)}]
Suppose that
\begin{align}\label{eq:mixed-relation}
F_0+\cdots+F_{k_0}=0,
\qquad F_k\in\mes_k^\CC\,,
\end{align}
and choose homogeneous elements $w_k\in\ha^{\geq2}_k\otimes_\QQ\CC$ with $G(w_k)=F_k$, where here and in the following we extend $G$ and $G^{(N)}$ $\CC$-linearly. Notice that all the convergence statements above also hold for these elements by linearity. Take $k_0$ maximal such that $F_{k_0}\neq0$. Evaluating \eqref{eq:mixed-relation} at $\tau_N=-1/z_N$ and multiplying by $\tau_N^{k_0}$ gives by the definition of $G^{(N)}$
\begin{align}
0
=G^{(N)}(w_{k_0})(\tau)
+
\sum_{k=0}^{k_0-1}
\tau_N^{k_0-k}G^{(N)}(w_k)(\tau)\,.
\end{align}
By \cref{cor:full-limit} all the functions $G^{(N)}(w_k)$ are locally bounded and the first term converges to $G(w_{k_0})(\tau)=F_{k_0}(\tau)$, while all the other terms tend to zero as $N\to\infty$. So $F_{k_0}(\tau)=0$ for all $\tau\in\HH$, which is a contradiction, and repeating this argument shows that every term in \eqref{eq:mixed-relation} vanishes. This proves the direct sum decomposition of $\mes^\CC$, and restricting to the rational spans gives the one for $\mes$. Since the harmonic product is homogeneous, we have $\mes_k\mes_\ell\subset\mes_{k+\ell}$, and therefore $W$ is a well-defined derivation.
\end{proof}

\subsection{Comparison with the weight filtration on
\texorpdfstring{$\mathsf{qMZV}$}{qMZV}}

We extend the assignment
\begin{align}
g:\ha^{\geq2}&\longrightarrow\mathsf{qMZV},\\
z_{k_1}\cdots z_{k_r}
&\longmapsto g_{k_1,\dots,k_r}
\end{align}
$\QQ$-linearly and set $g(1)=1$. To compare the Fourier expansions of $G_{k_1,\dots,k_r}$ and $g_{k_1,\dots,k_r}$, we put for $k\geq-1$
\begin{align}
V_k&\coloneqq\operatorname{Fil}^W_k\mathsf{qMZV},&
V_k^\CC&\coloneqq\Span_\CC(V_k)\subset\CC[[q]]\,.
\end{align}
If $w$ is homogeneous of weight $k$, then we have
\begin{align}\label{eq:triangular-Fourier}
(-2\pi i)^{-k}G(w)-g(w)
\in V_{k-1}^\CC\,.
\end{align}
This follows from \cite[Proposition~2.5]{BT} together with \cite[Lemma~2.4]{BT}, which shows that for $k_1,\dots,k_r\geq2$ every $g$ with an entry equal to $1$ cancels in the Fourier expansion of $G_{k_1,\dots,k_r}$ (see the proof of \cite[Theorem~1.1]{BT}). So the lower-weight part lies in $V_{k-1}^\CC$, which proves \eqref{eq:triangular-Fourier}. By induction on the weight we also obtain that $g(w)$ differs from $(-2\pi i)^{-k}G(w)$ by a complex linear combination of multiple Eisenstein series of weight smaller than $k$, which for indices with all entries at least $2$ is also the comparison given in \cite[Section~2.5]{HST}.

We will also need that for every $k$ the scalar extension gives an isomorphism
\begin{align}\label{eq:qMZV-rational-structure}
V_k\otimes_\QQ\CC
\xrightarrow{\,\sim\,}
V_k^\CC\,.
\end{align}
To see this, notice that $V_k$ is finite-dimensional and spanned by series with rational Fourier coefficients. After choosing a $\QQ$-basis of $V_k$, some finite set of Fourier coefficients gives a rational matrix of full column rank, and its rank does not change after extending scalars to $\CC$. Passing to the quotient gives an isomorphism
\begin{align}\label{eq:qMZV-graded-rational-structure}
\operatorname{gr}^W_k\mathsf{qMZV}\otimes_\QQ\CC
\xrightarrow{\,\sim\,}
V_k^\CC/V_{k-1}^\CC\,.
\end{align}
In particular,
\begin{align}\label{eq:qMZV-rational-intersection}
V_k\cap V_{k-1}^\CC=V_{k-1}\,.
\end{align}

\begin{proof}[Proof of \cref{thm:main-complex} \textup{(ii)}]
We set $\Phi_0(1)=1$ and define for $k\geq1$ and  $w\in\ha^{\geq2}_k$
\begin{align}
\Phi_k(G(w))
\coloneqq
g(w)
\bmod\operatorname{Fil}^W_{k-1}\mathsf{qMZV}\,.
\end{align}
If $G(w)=0$, then \eqref{eq:triangular-Fourier} shows that $g(w)\in V_{k-1}^\CC$, and since also $g(w)\in V_k$, we get $g(w)\in V_{k-1}$ by \eqref{eq:qMZV-rational-intersection}. This shows that $\Phi_k$ is well-defined, and it is surjective by the definition of the weight filtration. For the injectivity, suppose that $\Phi_k(G(w))=0$, i.e. $g(w)$ is a linear combination of elements of weight smaller than $k$. Writing each of these elements as a combination of multiple Eisenstein series as above, we obtain on the level of Fourier expansions
\begin{align}
G(w)\in\sum_{\ell<k}\mes_\ell^\CC\,.
\end{align}
Since all series involved converge for $|q|<1$, this inclusion also holds for the corresponding holomorphic functions on $\HH$. But $G(w)\in\mes_k^\CC$ and therefore the direct sum decomposition in part~\textup{(i)} gives $G(w)=0$, which shows that $\Phi_k$ is injective. It remains to show that $\Phi$ is compatible with the products. If $u$ and $v$ are homogeneous of weights $k$ and $\ell$, then multiplying their expressions \eqref{eq:triangular-Fourier} and using $G(u)G(v)=G(u\ast v)$ gives
\begin{align}
g(u)g(v)-g(u\ast v)
\in
V_{k+\ell-1}^\CC\,.
\end{align}
By \eqref{eq:qMZV-filtered-product} the left-hand side lies in $V_{k+\ell}$ and its image in $V_{k+\ell}^\CC/V_{k+\ell-1}^\CC$ vanishes, so by the injectivity of \eqref{eq:qMZV-graded-rational-structure} it already lies in $V_{k+\ell-1}$. This shows that the $\Phi_k$ give an isomorphism of graded $\QQ$-algebras
\begin{align}
\Phi:\mes\xrightarrow{\sim}\operatorname{gr}^W\mathsf{qMZV}\,.
\end{align}

Finally, let $F_1,\dots,F_d$ be a $\QQ$-basis of $\mes_k$. Then their images under $\Phi_k$ form a $\QQ$-basis of $\operatorname{gr}^W_k\mathsf{qMZV}$, which is also linearly independent over $\CC$ by \eqref{eq:qMZV-graded-rational-structure}. If $\sum_jc_jF_j=0$ with $c_j\in\CC$, then \eqref{eq:triangular-Fourier} gives
\begin{align}
\sum_jc_j\Phi_k(F_j)=0
\end{align}
after extending scalars to $\CC$, and therefore all $c_j$ vanish. This shows that the natural surjection $\mes_k\otimes_\QQ\CC\to\mes_k^\CC$ is injective, and taking the direct sum over all $k$ gives $\mes\otimes_\QQ\CC\cong\mes^\CC$.
\end{proof}

\section{The lowering operator and the
\texorpdfstring{$\mathfrak{sl}_2$}{sl2}-structure}
\label{sec:delta}

In this section, we want to complete the proof of \cref{thm:main-sl2}. For this, we combine \cref{thm:main-complex} \textup{(i)} with the two-term expansion in \cref{prop:two-term} to verify the criterion in \cref{prop:delta-equivalences}.

\begin{proof}[Proof of \cref{thm:main-sl2}]
By \cref{thm:main-complex} \textup{(i)}, the kernel of $G$ is spanned by homogeneous elements, so let $w\in\ker(G)$ be homogeneous. Then $G^{(N)}(w)=0$ for every $N$, and multiplying \eqref{eq:two-term} by $z_N$ and taking the limit $N\to\infty$ gives, since $|z_N|\asymp N$,
\begin{align}
G(\bar\delta w)=0\,.
\end{align}
So $\bar\delta(\ker(G))\subset\ker(G)$ and $\bar\delta$ induces a well-defined derivation $\delta$ on $\mes$. The well-definedness of $W$ is equivalent to $\ker(G)$ being homogeneous, and the one of $D$ was shown in \cite[Main Theorem~D]{BKM}, which proves (i). For the commutator relations, let $w\in\ha^{\geq2}$ be an arbitrary homogeneous element and let
\begin{align}
\pi^{\mathrm f}:\ha^{\geq2}_\ast&\twoheadrightarrow\mes^{\mathrm f},\\
z_{k_1}\cdots z_{k_r}
&\longmapsto G^{\mathrm f}(k_1,\dots,k_r)
\end{align}
be the canonical surjective algebra homomorphism with $\pi^{\mathrm f}(1)=1$. For the map $\rho$ in \eqref{eq:formal-realization} we then have
\begin{align}\label{eq:formal-analytic-realization}
\rho\circ\pi^{\mathrm f}=G\,.
\end{align}

The algebra $\mathcal{G}^{\mathrm f}$ is equipped with the $\sltwo$-triple $(W^{\mathrm f},D^{\mathrm f},\delta^{\mathrm f})$ constructed in \cite[Theorem~4.12]{BIM}, and by the explicit formula in \cite[Equation~(4.7), in the proof of Proposition~4.15]{BIM}, its weight and lowering operators satisfy on the image of $\ha^{\geq2}$
\begin{align}\label{eq:formal-weight-lowering}
W^{\mathrm f}\pi^{\mathrm f}(w)
=\pi^{\mathrm f}(\bar Ww),
\qquad
\delta^{\mathrm f}\pi^{\mathrm f}(w)
=\pi^{\mathrm f}(\bar\delta w)\,,
\end{align}
where $\bar{W}(w)=\wt(w)w$. By \cite[Main Theorem~D]{BKM}, there is a homogeneous map $\theta:\ha^{\geq2}\to\ha^{\geq2}$ of weight $2$ such that
\begin{align}\label{eq:theta-realization}
D^{\mathrm f}\pi^{\mathrm f}(w)
=\pi^{\mathrm f}(\theta w),
\qquad
2\pi i\frac{d}{d\tau}G(w)
=G(\theta w)\,.
\end{align}
The commutator relation for the formal operators restricted to $\mes^{\mathrm f}$ therefore gives
\begin{align}
\pi^{\mathrm f}
\bigl(([\bar\delta,\theta]-\bar W)w\bigr)
&=
\bigl([\delta^{\mathrm f},D^{\mathrm f}]-W^{\mathrm f}\bigr)
\pi^{\mathrm f}(w)
=0\,.
\end{align}
Applying $\rho$ and using \eqref{eq:formal-analytic-realization}, we get
\begin{align}\label{eq:realized-commutator}
G([\bar\delta,\theta]w)=G(\bar Ww)\,,
\end{align}
and therefore
\begin{align}
[\delta,D]G(w)
=G([\bar\delta,\theta]w)
=G(\bar Ww)
=WG(w)\,.
\end{align}
Since the elements $G(w)$ span $\mes$, this proves $[\delta,D]=W$ on $\mes$, and the other two commutator relations in \eqref{eq:sl2-relations} follow since $D$ raises the weight by $2$ and $\delta$ lowers the weight by $2$. This proves (ii). Notice that \eqref{eq:formal-analytic-realization}--\eqref{eq:theta-realization} also show that $\rho$ commutes with all three operators, and together with \cite[Main Theorem~C]{BKM} this shows that $\rho:\mes^{\mathrm f}\twoheadrightarrow\mes$ is a surjective morphism of $\sltwo$-algebras, where no injectivity is needed. Part~\textup{(iii)} follows from \cref{prop:delta-equivalences}.
\end{proof}

\section{Proof of the first-order estimate}
\label{sec:proof-first-order}
In this section, we give the proof of \cref{thm:first-order}. Notice that we will not use any result from \Cref{sec:weight,sec:delta} here. First, we fix some notation.

\begin{enumerate}[(i)]
\item We denote by $\succ_\infty=\succ$ the total order defined in \eqref{eq:mes order} and denote $G^{(\infty)}_\kk(\tau)=G_\kk(\tau)$.
\item Let $X$ be a topological set. For functions $f:X\rightarrow \CC$, $g:X\rightarrow\RR_{\geq0}$ and $S\subset X$,
\begin{align}
&f(x)\ll g(x)\quad\text{for $x\in S$}\\
&\equivdef\text{There exists a real number $C>0$ such that for any $x\in S$, $|f(x)|\leq Cg(x)$.}
\end{align}
\item The constant $C$ above is called the \emph{implicit constant} or \emph{implied constant}.
\item If the implicit constant depends on other variables $a_1,\dots,a_n$, we write $f(x)\ll_{a_1,\dots,a_n}g(x)$ to indicate this dependence.
\end{enumerate}

For $\nu\in\ZZ_{>0}\cup\{\infty\}$, $\mu_1,\mu_2\in\Lambda_\tau$ and $k_1,\dots,k_r\geq2$, we define the truncated version of multiple Eisenstein series by
\begin{align}
G^{(\nu;\mu_1,\mu_2)}_{k_1,\dots,k_r}(\tau)\coloneqq\begin{cases}\summ{\lambda_1,\dots,\lambda_r\in\latzero\\\mu_1\succ_\nu\lambda_1\succ_\nu\cdots\succ_\nu\lambda_r\succ_\nu\mu_2}\frac{1}{\lambda_1^{k_1}\cdots\lambda_r^{k_r}}&r\geq1,\\\ind_{\mu_1\succ_\nu\mu_2}&r=0.\end{cases}
\end{align}
The series in the right-hand side converges absolutely and uniformly for any compact subset of $\HH$. (\cref{lem:trunc conv})

\subsection{Estimates for various Eisenstein type series}

In this subsection, we fix a compact subset $C\subset\HH$ and a real number $m\geq0$.

\begin{lem}\label{lem:single eis1}
Let $\nu\in\ZZ_{>0}\cup\{\infty\}$ and $k\geq2$. For $\tau\in C$ and lattice points $\mu_1,\mu_2\in\Lambda_\tau$ with $\mu_1\succ_\nu\mu_2$, we have
\begin{align}
\summ{\lambda\in\latzero\\\mu_1\succ_\nu\lambda\succ_\nu\mu_2}\frac{\LL(|\lambda|)^m}{|\lambda|^k}\ll_{C,k,m}\begin{cases}\LL(|\mu_1|+|\mu_2|)^{m+1}&k=2,\\1&k\geq3.\end{cases}
\end{align}
\end{lem}

\begin{proof}
It is known that the series $\textstyle\summ{\lambda\in\latzero}\frac{\LL(|\lambda|)^m}{|\lambda|^k}$ converges uniformly on $\tau\in C$ for the case $k\geq3$ (cf. \cite[p.8 Exercise 1.1.4]{DS}). For the case $k=2$, we divide the sum as follows:
\begin{align}\label{eq:two adic}
\summ{\lambda\in\latzero\\\mu_1\succ_\nu\lambda\succ_\nu\mu_2}\frac{\LL(|\lambda|)^m}{|\lambda|^2}=\summ{\mu_1\succ_\nu\lambda\succ_\nu\mu_2\\0<|\lambda|<1}\frac{\LL(|\lambda|)^m}{|\lambda|^2}
+\sum_{j=0}^\infty\summ{\mu_1\succ_\nu\lambda\succ_\nu\mu_2\\2^j\leq|\lambda|<2^{j+1}}\frac{\LL(|\lambda|)^m}{|\lambda|^2}\,.
\end{align}
Since $C$ is compact, $m_C\coloneqq\inf_{\tau\in C}\min_{\lambda\in\latzero}|\lambda|$ exists and $m_C>0$. Furthermore, for any $\tau\in C$ and any distinct lattice points $\lambda,\lambda^\prime\in\Lambda_\tau$, their distance satisfies $|\lambda-\lambda^\prime|\geq m_C$. Now, we give an estimate for the first term. We have
\begin{align}
\summ{\mu_1\succ_\nu\lambda\succ_\nu\mu_2\\0<|\lambda|<1}\frac{\LL(|\lambda|)^m}{|\lambda|^2}\leq\summ{\mu_1\succ_\nu\lambda\succ_\nu\mu_2\\m_C\leq|\lambda|<1}\frac{\LL(1)^m}{m_C^2}\,.
\end{align}
The number of lattice points in the running indices can be estimated as follows
\begin{align}
\#\{\lambda\in\Lambda_\tau\mid m_C\leq|\lambda|<1\}\cdot\pi\left(\frac{m_C}{4}\right)^2\leq\pi\left(1+\frac{m_C}{4}\right)^2\,.
\end{align}
Therefore, the first term of \eqref{eq:two adic} is bounded by a constant that does not depend on $\tau\in C$. For the second term, similar arguments yield the following:
\begin{align}
\#\{\lambda\in\Lambda_\tau\mid\mu_1\succ_\nu\lambda\succ_\nu\mu_2,R\leq|\lambda|<2R\}\cdot\pi\left(\frac{m_C}{4}\right)^2\ll_C\min\{R^2,(w+1)R+1\}
\end{align}
for any $R>0$, where $w$ is the width between the straight lines $\mu_1+\RR(\nu-\tau)$ and $\mu_2+\RR(\nu-\tau)$ (when $\nu=\infty$, $\mu+\RR(\nu-\tau)$ means the horizontal line $\mu+\RR$). Thus, we have
\begin{align}
\summ{\mu_1\succ_\nu\lambda\succ_\nu\mu_2\\R\leq|\lambda|<2R}\frac{\LL(|\lambda|)^m}{|\lambda|^2}
&\ll_C\min\{R^2,(w+1)R+1\}\frac{\LL(2R)^m}{R^2}\\
&\ll_{C,m}\left(R^2\min\left\{1,\frac{w}{R}\right\}+R+1\right)\frac{\LL(R)^m}{R^2}\\
&\ll_{C,m}\begin{cases}\LL(R)^m&R\leq w,\\\frac{w+1}{R}\LL(R)^m&R>w,\end{cases}
\end{align}
and thus, the second term of \eqref{eq:two adic} can be estimated as follows:
\begin{align}
\sum_{j=0}^\infty\summ{\mu_1\succ_\nu\lambda\succ_\nu\mu_2\\2^j\leq|\lambda|<2^{j+1}}\frac{\LL(|\lambda|)^m}{|\lambda|^2}&=\sum_{j=0}^\infty\summ{\mu_1\succ_\nu\lambda\succ_\nu\mu_2\\2^j\leq|\lambda|<2^{j+1}\\2^j\leq w}\frac{\LL(|\lambda|)^m}{|\lambda|^2}+\sum_{j=0}^\infty\summ{\mu_1\succ_\nu\lambda\succ_\nu\mu_2\\2^j\leq|\lambda|<2^{j+1}\\2^j>w}\frac{\LL(|\lambda|)^m}{|\lambda|^2}\\
&\ll_{C,m}\sum_{0\leq j\leq\log_2w}\LL(2^j)^m+\sum_{j\geq0}\frac{w+1}{2^{j_0+j}}\LL(2^{j_0+j})^m\,,
\end{align}
where $j_0\coloneqq\min\{j\geq0\mid 2^j>w\}$. Since it holds $\LL(2^j)=1+\log(2+2^j)\ll1+j$, $\LL(2^{j_0+j})\leq\LL(2^{j_0})+j$ for any $j\geq0$, and $\frac{w+1}{2^{j_0}}\ll1$ we have
\begin{align}
\sum_{0\leq j\leq\log_2w}\LL(2^j)^m+\sum_{j\geq0}\frac{w+1}{2^{j_0+j}}\LL(2^{j_0+j})^m\ll\sum_{0\leq j\leq\log_2w}(1+j)^m+\sum_{j\geq0}\frac{(\LL(2^{j_0})+j)^m}{2^j}\,.
\end{align}
Now, we have
\begin{align}
\sum_{0\leq j\leq\log_2w}(1+j)^m\leq\sum_{0\leq j\leq\log_2w}(1+\log_2w)^m\leq(1+\log_2w)^{m+1}\ll\LL(w)^{m+1},
\end{align}
and
\begin{align}
\sum_{j\geq0}\frac{(\LL(2^{j_0})+j)^m}{2^j}&\ll_m\sum_{j\geq0}\frac{\LL(2^{j_0})^m+j^m}{2^j}\\
&\ll_m\LL(2^{j_0})^m+1\\
&\ll_m\LL(2^{j_0})^m.
\end{align}
Since it holds $2^{j_0-1}\leq w$ if $w\geq1$ and $2^{j_0}=1$ if $w<1$ by definition of $j_0$, we have $\LL(2^{j_0})\ll\LL(2w)\ll\LL(w)$. Finally, we can estimate the second term as follows
\begin{align}
\sum_{j=0}^\infty\summ{\mu_1\succ_\nu\lambda\succ_\nu\mu_2\\2^j\leq|\lambda|<2^{j+1}}\frac{\LL(|\lambda|)^m}{|\lambda|^2}\ll_{C,m}\LL(w)^{m+1}\leq\LL(|\mu_1|+|\mu_2|)^{m+1}
\end{align}
since the width $w$ is bounded by $w\leq|\mu_1-\mu_2|\leq|\mu_1|+|\mu_2|$.
\end{proof}

\begin{lem}\label{lem:trunc conv}
For any $\nu\in\ZZ_{>0}\cup\{\infty\}$, $\mu_1,\mu_2\in\Lambda_\tau$, $r\geq1$ and $k_1,\dots,k_r\geq2$, the series
\begin{align}
\summ{\lambda_1,\dots,\lambda_r\in\latzero\\\mu_1\succ_\nu\lambda_1\succ_\nu\cdots\succ_\nu\lambda_r\succ_\nu\mu_2}\frac{1}{\lambda_1^{k_1}\cdots\lambda_r^{k_r}}
\end{align}
converges absolutely and uniformly for any compact subset of $\HH$.
\end{lem}
\begin{proof}
By \cref{lem:single eis1}, the series $\summ{\lambda\in\latzero\\\mu_1\succ_\nu\lambda\succ_\nu\mu_2}\frac{1}{|\lambda|^k}$ converges uniformly for $\tau\in C$ for any $k\geq2$. Therefore, the series
\begin{align}
\summ{\lambda_1,\dots,\lambda_r\in\latzero\\\mu_1\succ_\nu\lambda_1\succ_\nu\cdots\succ_\nu\lambda_r\succ_\nu\mu_2}\frac{1}{|\lambda_1|^{k_1}\cdots|\lambda_r|^{k_r}}\leq\prod_{i=1}^r\summ{\lambda_i\in\latzero\\\mu_1\succ_\nu\lambda_i\succ_\nu\mu_2}\frac{1}{|\lambda_i|^{k_i}}
\end{align}
also converges uniformly for $\tau\in C$.
\end{proof}

\begin{lem}\label{lem:single eis2}
Let $k\geq3$. For $\tau\in C$ and $R>0$, we have
\begin{align}
\summ{\mu\in\latzero\\|\mu|>R}\frac{\LL(|\mu|)^m}{|\mu|^k}\ll_{C,k,m}\frac{\LL(R)^m}{R^{k-2}},\label{eq:single eis big}\\
\summ{\mu\in\latzero\\0<|\mu|\leq R}\frac{\LL(|\mu|)^m}{|\mu|^2}\ll_{C,m}\LL(R)^{m+1}\,.\label{eq:single eis small}
\end{align}
\end{lem}

\begin{proof}
For $0<R<1$, both series are bounded by a constant from the above, and $1\ll_m\frac{\LL(R)^m}{R^{k-2}},\LL(R)^{m+1}$. Thus, it is enough to consider for $R\geq1$. First, it holds
\begin{align}
\#\{\mu\in\Lambda_\tau\mid|\mu|\leq R\}\ll_C R^2\,.
\end{align}
Indeed, for $\mu=a\tau+b$, we have $\sqrt{a^2+b^2}\leq d_C|a\tau+b|$ for some constant $d_C$ since $C$ is compact. Thus,
\begin{align}
\#\{\mu\in\Lambda_\tau\mid|\mu|\leq R\}&=\#\{(a,b)\in\ZZ^2\mid |a\tau+b|^2\leq R^2\}\\
&\leq\#\{(a,b)\in\ZZ^2\mid a^2+b^2\leq d_C^2R^2\}\ll_CR^2\,.
\end{align}
Therefore, the series in \eqref{eq:single eis big} can be estimated as follows:
\begin{align}
\summ{|\mu|>R}\frac{\LL(|\mu|)^m}{|\mu|^k}
&=\sum_{j=0}^\infty\summ{2^jR<|\mu|\leq2^{j+1}R}\frac{\LL(|\mu|)^m}{|\mu|^k}\\
&\leq\sum_{j=0}^\infty\#\{\mu\in\Lambda_\tau\mid2^jR<|\mu|\leq2^{j+1}R\}\frac{\LL(2^{j+1}R)^m}{(2^jR)^k}\\
&\ll_C\sum_{j=0}^\infty(2^{j+1}R)^2\frac{\LL(2^{j+1}R)^m}{(2^jR)^k}\\
&\ll\frac{1}{R^{k-2}}\sum_{j=0}^\infty\frac{\LL(2^{j+1}R)^m}{2^{j(k-2)}}\,.
\end{align}
Since $\LL(2^{j+1}R)\leq\LL(R)+(j+1)\log2\ll(j+1)\LL(R)$, we have
\begin{align}
\sum_{j=0}^\infty\frac{\LL(2^{j+1}R)^m}{2^{j(k-2)}}\ll_m\LL(R)^m\sum_{j=0}^\infty\frac{(j+1)^m}{2^{j(k-2)}}\ll_{m}\LL(R)^m\,.
\end{align}
Now, we prove \eqref{eq:single eis small}. Let $i_0\coloneqq\min\{i\geq0\mid R\leq2^{i+1}\}$, then we have
\begin{align}
\summ{0<|\mu|\leq R}\frac{\LL(|\mu|)^m}{|\mu|^2}&\leq\left(\summ{m_C\leq|\mu|\leq 1}+\sum_{i=0}^{i_0}\summ{2^i<|\mu|\leq 2^{i+1}}\right)\frac{\LL(|\mu|)^m}{|\mu|^2}\\
&\ll_{C,m}\frac{\LL(1)^m}{m_C^2}+\sum_{i=0}^{i_0}\#\{\mu\in\Lambda_\tau\mid2^i<|\mu|\leq2^{i+1}\}\frac{\LL(2^{i+1})^m}{2^{2i}}\,,
\end{align}
where $m_C=\inf_{\tau\in C}\min_{\lambda\in\latzero}|\lambda|>0$ as defined in the proof of \cref{lem:single eis1}. Again, the summand can be estimated as follows
\begin{align}
\#\{\mu\in\Lambda_\tau\mid2^i<|\mu|\leq2^{i+1}\}\frac{\LL(2^{i+1})^m}{2^{2i}}\ll_m(i+1)^m\,.
\end{align}
Thus, we have
\begin{align}
\summ{0<|\mu|\leq R}\frac{\LL(|\mu|)^m}{|\mu|^2}\ll_{C,m}1+\sum_{i=0}^{i_0}(i+1)^m\leq1+(i_0+1)^{m+1}\,.
\end{align}
Now, it holds $i_0+1\ll\LL(R)$ by $i_0\leq\log_2R$, and therefore we have the desired estimate.
\end{proof}

\begin{lem}\label{lem:shifted single eis}
Let $k_1,k_2\geq2$. For $\tau\in C$ and $\lambda\in\latzero$, we have
\begin{align}\label{eq:shifted single eis}
\summ{\mu\in\Lambda_\tau\setminus\{0,\lambda\}}\frac{\LL(|\mu|)^m}{|\mu|^{k_1}|\mu-\lambda|^{k_2}}
\ll_{C,m}\frac{\LL(|\lambda|)^{m+1}}{|\lambda|^2}\,.
\end{align}
\end{lem}

\begin{proof}
First, we prove \eqref{eq:shifted single eis} for $|\lambda|<2$. In this case, we have $|\mu|,|\mu-\lambda|\geq m_C$ for any $\mu\in\Lambda_\tau\setminus\{0,\lambda\}$ and $|\mu-\lambda|\geq\frac{|\mu|}{2}$ if $|\mu|\geq4$. Thus, we have
\begin{align}
\summ{\mu\in\Lambda_\tau\setminus\{0,\lambda\}}\frac{\LL(|\mu|)^m}{|\mu|^{k_1}|\mu-\lambda|^{k_2}}&=\left(\summ{|\mu|<4}+\summ{|\mu|\geq4}\right)\frac{\LL(|\mu|)^m}{|\mu|^{k_1}|\mu-\lambda|^{k_2}}\\
&\ll_{C,m}1+\summ{|\mu|\geq4}\frac{\LL(|\mu|)^m}{|\mu|^{k_1+k_2}}\ll_{C,m}1\,.
\end{align}
On the other hand, since $m_C\leq|\lambda|<2$, we have
\begin{align}
\frac{\LL(|\lambda|)^m}{|\lambda|^2}\geq\frac{\LL(m_C)^m}{4}\gg_{C}1\,,
\end{align}
and thus, we have \eqref{eq:shifted single eis} for $|\lambda|<2$. Now, we prove the case for $|\lambda|\geq2$. We divide the series in \eqref{eq:shifted single eis} into four terms as follows:
\begin{align}
\summ{\mu\in\Lambda_\tau\setminus\{0,\lambda\}}\frac{\LL(|\mu|)^m}{|\mu|^{k_1}|\mu-\lambda|^{k_2}}=\sum_{\mu\in I_0\cup I_1\cup I_2\cup I_3}\frac{\LL(|\mu|)^m}{|\mu|^{k_1}|\mu-\lambda|^{k_2}}\,,
\end{align}
where
\begin{align}
I_0&\coloneqq\left\{\mu\in\Lambda_\tau\setminus\{0,\lambda\}\relmid|\mu|\leq\frac{|\lambda|}{2}\right\},&
I_1&\coloneqq\left\{\mu\in\Lambda_\tau\setminus\{0,\lambda\}\relmid|\mu-\lambda|\leq\frac{|\lambda|}{2}\right\},\\
I_2&\coloneqq\left\{\mu\in\Lambda_\tau\setminus\{0,\lambda\}\relmid\frac{|\lambda|}{2}<|\mu|\leq|\mu-\lambda|\right\},&
I_3&\coloneqq\left\{\mu\in\Lambda_\tau\setminus\{0,\lambda\}\relmid\frac{|\lambda|}{2}<|\mu-\lambda|<|\mu|\right\}\,.
\end{align}
For the first term, by \cref{lem:single eis2}, we have
\begin{align}
\sum_{\mu\in I_0}\frac{\LL(|\mu|)^m}{|\mu|^{k_1}|\mu-\lambda|^{k_2}}&\ll\frac{1}{|\lambda|^{k_2}}\summ{0<|\mu|\leq\frac{|\lambda|}{2}}\frac{\LL(|\mu|)^m}{|\mu|^{k_1}}\\
&\ll_{C,m}\frac{\LL(|\lambda|)^{m+1}}{|\lambda|^{k_2}}\,.
\end{align}
The second term can also be estimated as follows:
\begin{align}
\sum_{\mu\in I_1}\frac{\LL(|\mu|)^m}{|\mu|^{k_1}|\mu-\lambda|^{k_2}}&\ll\frac{\LL(|\lambda|)^m}{|\lambda|^{k_1}}\summ{\mu\in\latzero\\0<|\mu-\lambda|\leq\frac{|\lambda|}{2}}\frac{1}{|\mu-\lambda|^{k_2}}\\
&\ll_{C,m}\frac{\LL(|\lambda|)^{m+1}}{|\lambda|^{k_1}}\,.
\end{align}
For the third term, using \eqref{eq:single eis big}, we have
\begin{align}
\sum_{\mu\in I_2}\frac{\LL(|\mu|)^m}{|\mu|^{k_1}|\mu-\lambda|^{k_2}}\leq\summ{|\mu|>\frac{|\lambda|}{2}}\frac{\LL(|\mu|)^m}{|\mu|^{k_1+k_2}}\ll_{C,k_1,k_2,m}\frac{\LL(|\lambda|)^m}{|\lambda|^{k_1+k_2-2}}\,.
\end{align}
For the fourth term, if $\mu\in I_3$, it holds $|\mu|\leq|\mu-\lambda|+|\lambda|<3|\mu-\lambda|$. Therefore we have
\begin{align}
\sum_{\mu\in I_3}\frac{\LL(|\mu|)^m}{|\mu|^{k_1}|\mu-\lambda|^{k_2}}\ll\summ{|\mu-\lambda|>\frac{|\lambda|}{2}}\frac{\LL(|\mu-\lambda|)^m}{|\mu-\lambda|^{k_1+k_2}}\ll_{C,k_1,k_2,m}\frac{\LL(|\lambda|)^m}{|\lambda|^{k_1+k_2-2}}\,.
\end{align}
Combining the estimates, we have \eqref{eq:shifted single eis} for $|\lambda|\geq2$.
\end{proof}

Denote $\chi_N(\lambda)=\ind_{\lambda\ordN0}-\ind_{\lambda\succ0}$ and, for $j\in\ZZ_{>0}$, $f_j(\mu,\lambda)=(\mu-\lambda)^{-2j}$ if $\mu\neq\lambda$ and $f_j(\mu,\lambda)=0$ if $\mu=\lambda$.

\begin{lem}\label{lem:single diff}
Let $\mu_1,\mu_2$ be lattice points that satisfy $\mu_1\ordN\mu_2$ and $\mu_1\succ\mu_2$. For $N,j\in\ZZ_{>0}$, we have
\begin{align}\label{eq:single diff}
G^{(N;\mu_1,\mu_2)}_{2j}(\tau)-G^{(\infty;\mu_1,\mu_2)}_{2j}(\tau)=\summ{\lambda\in\latzero}\chi_N(\lambda)(f_j(\mu_1,\lambda)-f_j(\mu_2,\lambda))\,,
\end{align}
and the series in the right-hand side absolutely converges for any $j\geq1$.
\end{lem}

\begin{proof}
We first prove the absolute convergence of the right-hand side. For sufficiently large $|\lambda|$, we have
\begin{align}
f_j(\mu_1,\lambda)-f_j(\mu_2,\lambda)&=\frac{(\mu_2-\lambda)^{2j}-(\mu_1-\lambda)^{2j}}{(\mu_1-\lambda)^{2j}(\mu_2-\lambda)^{2j}}\\
&=(\mu_2-\mu_1)\frac{\sum_{i=0}^{2j-1}(\mu_2-\lambda)^{2j-i-1}(\mu_1-\lambda)^{i}}{(\mu_1-\lambda)^{2j}(\mu_2-\lambda)^{2j}}\,.
\end{align}
Since it holds
\begin{align}
\frac{1}{2}|\lambda|\leq|\lambda|-|\mu_i|\leq|\mu_i-\lambda|\leq|\mu_i|+|\lambda|\leq\frac{3}{2}|\lambda|\quad(i=1,2)
\end{align}
if $|\lambda|\geq2\max\{|\mu_1|,|\mu_2|\}$, we have
\begin{align}
|f_j(\mu_1,\lambda)-f_j(\mu_2,\lambda)|\ll_j|\mu_2-\mu_1|\frac{|\lambda|^{2j-1}}{|\lambda|^{4j}}=\frac{|\mu_2-\mu_1|}{|\lambda|^{2j+1}}
\end{align}
for sufficiently large $|\lambda|$. Therefore the right-hand side absolutely converges. Now, we prove \eqref{eq:single diff}. Since $G^{(\nu;\mu_1,\mu_2)}_{2j}(\tau)$ absolutely converges for any $j\geq1$, we have
\begin{align}
G^{(\nu;\mu_1,\mu_2)}_{2j}(\tau)&=\lim_{M\to\infty}\summ{0\neq\lambda\in\ZZ_M\tau+\ZZ_M\\\mu_1\succ_\nu\lambda\succ_\nu\mu_2}\frac{1}{\lambda^{2j}}\\
&=\lim_{M\to\infty}\left(\summ{\lambda\in\ZZ_M\tau+\ZZ_M\\\mu_1\succ_\nu\lambda}\frac{\ind_{\lambda\neq0}}{\lambda^{2j}}-\summ{\lambda\in\ZZ_M\tau+\ZZ_M\\\mu_2\succeq_\nu\lambda}\frac{\ind_{\lambda\neq0}}{\lambda^{2j}}\right)\\
&=\lim_{M\to\infty}\left(\summ{\lambda\in\mu_1+\ZZ_M\tau+\ZZ_M}\ind_{\lambda\succ_\nu0}f_j(\mu_1,\lambda)-\summ{\lambda\in\mu_2+\ZZ_M\tau+\ZZ_M}\ind_{\lambda\succ_\nu0}f_j(\mu_2,\lambda)-f_j(\mu_2,0)\right)\\
&=\lim_{M\to\infty}\left(\summ{\lambda\in\ZZ_M\tau+\ZZ_M}\ind_{\lambda\succ_\nu0}f_j(\mu_1,\lambda)-\summ{\lambda\in\ZZ_M\tau+\ZZ_M}\ind_{\lambda\succ_\nu0}f_j(\mu_2,\lambda)-f_j(\mu_2,0)\right)
\end{align}
The last equality follows from
\begin{align}
\lim_{M\to\infty}\left(\summ{\lambda\in\mu+\ZZ_M\tau+\ZZ_M}-\summ{\lambda\in\ZZ_M\tau+\ZZ_M}\right)\ind_{\lambda\succ_\nu0}f_j(\mu,\lambda)=0
\end{align}
for $\nu\in\ZZ_{>0}\cup\{\infty\}$ and $\mu\in\Lambda_\tau$. Indeed, we have
\begin{align}
\left|\left(\summ{\lambda\in\mu+\ZZ_M\tau+\ZZ_M}-\summ{\lambda\in\ZZ_M\tau+\ZZ_M}\right)\ind_{\lambda\succ_\nu0}f_j(\mu,\lambda)\right|\leq\sum_{\lambda\in D_{M,\mu}}|f_j(\mu,\lambda)|
\end{align}
where $D_{M,\mu}\coloneqq(\mu+\ZZ_M\tau+\ZZ_M)\triangle(\ZZ_M\tau+\ZZ_M)$ is the symmetric difference of $\mu+\ZZ_M\tau+\ZZ_M$ and $\ZZ_M\tau+\ZZ_M$. It holds $\#D_{M,\mu}=2(2M-1)(|a|+|b|)-2|a||b|\ll_{\mu}M$ where $\mu=a\tau+b$, and for sufficiently large $M$, we have $|\mu-\lambda|\gg_{\mu,\tau} M$ for any $\lambda\in D_{M,\mu}$. Therefore, we have
\begin{align}
\sum_{\lambda\in D_{M,\mu}}|f_j(\mu,\lambda)|\ll_{\mu,\tau}M^{1-2j}
\end{align}
for sufficiently large $M$. Thus, the difference $G^{(N;\mu_1,\mu_2)}_{2j}(\tau)-G^{(\infty;\mu_1,\mu_2)}_{2j}(\tau)$ is equal to the right-hand side of \eqref{eq:single diff}.
\end{proof}

\begin{lem}\label{lem:chiN}
For $\tau\in C$ and $N\in\ZZ_{>0}$, we have
\begin{align}\label{eq:chiN}
\summ{\lambda\in\latzero\\\chi_N(\lambda)\neq0}\frac{\LL(|\lambda|)^m}{|\lambda|^3}\ll_{C,m}\frac{\LL(N)^m}{N^2}\,.
\end{align}
\end{lem}

\begin{proof}
First, the support of $\chi_N$ can be determined as follows:
\begin{align}
\chi_N(\lambda)\neq0\equiviff\lambda\in A_N\cup-A_N\,,
\end{align}
where $A_N\coloneqq\{\lambda\in\Lambda_\tau\mid\lambda\prec_N0,\lambda\succ0\}=\{a\tau+b\in\Lambda_\tau\mid a\geq1,b\leq-Na\}$. Thus, we have
\begin{align}
\summ{\lambda\in\latzero\\\chi_N(\lambda)\neq0}\frac{\LL(|\lambda|)^m}{|\lambda|^3}&=2\sum_{\lambda\in A_N}\frac{\LL(|\lambda|)^m}{|\lambda|^3}
=2\sum_{a\geq1,b\geq0}\frac{\LL(|a(\tau-N)-b|)^m}{|a(\tau-N)-b|^3}
\end{align}
Now, we prove \eqref{eq:chiN} for $N\geq2M_C\coloneqq2\sup_{\tau\in C}|\Re(\tau)|$. Notice that it holds
\begin{align}
aN+b\ll_C|a(\tau-N)-b|\ll_CaN+b\,.
\end{align}
Indeed, let $M_C^\prime\coloneqq\sup_{\tau\in C}|\tau|$, then we have
\begin{align}
&|\Re(a(\tau-N)-b)|=a(N-\Re(\tau))+b\geq a(N-M_C)+b\geq\frac{1}{2}(aN+b),\\
&|a(\tau-N)-b|\leq a|\tau-N|+b\leq a(N+M_C^\prime)+b\leq(1+M_C^\prime)(aN+b)\,.
\end{align}
Therefore, we have
\begin{align}
\sum_{a\geq1,b\geq0}\frac{\LL(|a(\tau-N)-b|)^m}{|a(\tau-N)-b|^3}
&=\sum_{j=0}^\infty\summ{a\geq1,b\geq0\\2^jN\leq aN+b<2^{j+1}N}\frac{\LL(|a(\tau-N)-b|)^m}{|a(\tau-N)-b|^3}\\
&\ll_C\sum_{j=0}^\infty\#\left\{(a,b)\in\ZZ^2\relmid\begin{matrix}a\geq1,b\geq0\\2^jN\leq aN+b<2^{j+1}N\end{matrix}\right\}\frac{\LL(2^{j+1}N)^m}{(2^jN)^3}\\
&\ll\sum_{j=0}^\infty2^{2j}N\frac{\LL(2^{j+1}N)^m}{(2^jN)^3}\,.
\end{align}
This series can be estimated as follows:
\begin{align}
\sum_{j=0}^\infty2^{2j}N\frac{\LL(2^{j+1}N)^m}{(2^jN)^3}&\ll_m\frac{1}{N^2}\sum_{j=0}^\infty\frac{\LL(N)^m(j+1)^m}{2^j}\ll_m\frac{\LL(N)^m}{N^2}\,.
\end{align}
This shows \eqref{eq:chiN} for $N\geq2M_C$. For $1\leq N<2M_C$, we have
\begin{align}
\summ{\lambda\in\Lambda_\tau\\\chi_N(\lambda)\neq0}\frac{\LL(|\lambda|)^m}{|\lambda|^3}\ll_{C,m}1\ll_{C,m}\frac{\LL(N)^m}{N^2}\,.
\end{align}
This completes the proof.
\end{proof}

\begin{lem}\label{lem:key estim1}
Let $j\in\ZZ_{>0}$ and $k,k_1,k_2\geq3$. For $\tau\in C$ and $\lambda\in\latzero$, we have
\begin{align}
&\summ{\mu_1,\mu_2\in\latzero}\frac{\LL(|\mu_1|+|\mu_2|)^m}{|\mu_1|^{k_1}|\mu_2|^{k_2}}|f_j(\mu_1,\lambda)-f_j(\mu_2,\lambda)|\ll_{C,m}\frac{\LL(|\lambda|)^{2m+1}}{|\lambda|^3},\label{eq:key estim double}\\
&\summ{\mu\in\latzero}\frac{\LL(|\mu|)^m}{|\mu|^k}|f_j(\mu,\lambda)-f_j(0,\lambda)|\ll_{C,m}\frac{\LL(|\lambda|)^{2m+1}}{|\lambda|^3}\,.\label{eq:key estim single}
\end{align}
\end{lem}

\begin{proof}
We first prove \eqref{eq:key estim double}. We divide the series in \eqref{eq:key estim double} into three terms as follows:
\begin{align}
J_0&\coloneqq\left\{(\mu_1,\mu_2)\in(\latzero)^2\relmid|\mu_1|,|\mu_2|\leq\frac{|\lambda|}{2}\right\},\\
J_1&\coloneqq\left\{(\mu_1,\mu_2)\in(\latzero)^2\relmid|\mu_1|>\frac{|\lambda|}{2}\right\},\\
J_2&\coloneqq\left\{(\mu_1,\mu_2)\in(\latzero)^2\relmid|\mu_1|\leq\frac{|\lambda|}{2}<|\mu_2|\right\}\,.
\end{align}
For the first term, since it holds
\begin{align}
|f_j(\mu_1,\lambda)-f_j(\mu_2,\lambda)|\ll_j\frac{|\mu_1-\mu_2|}{|\lambda|^{2j+1}}\leq\frac{|\mu_1|+|\mu_2|}{|\lambda|^{2j+1}}\,,
\end{align}
we have
\begin{align}
&\sum_{(\mu_1,\mu_2)\in J_0}\frac{\LL(|\mu_1|+|\mu_2|)^m}{|\mu_1|^{k_1}|\mu_2|^{k_2}}|f_j(\mu_1,\lambda)-f_j(\mu_2,\lambda)|\\
&\ll_j\frac{1}{|\lambda|^{2j+1}}\sum_{(\mu_1,\mu_2)\in J_0}\left(\frac{\LL(|\mu_1|)^m\LL(|\mu_2|)^m}{|\mu_1|^{k_1-1}|\mu_2|^{k_2}}+\frac{\LL(|\mu_1|)^m\LL(|\mu_2|)^m}{|\mu_1|^{k_1}|\mu_2|^{k_2-1}}\right)\\
&\ll_{C,m}\frac{\LL(|\lambda|)^{m+1}}{|\lambda|^{2j+1}}
\end{align}
For the last inequality, we use \eqref{eq:single eis small} and the fact that $\textstyle\summ{\mu\in\latzero}\frac{\LL(|\mu|)^m}{|\mu|^k}$ converges for $k\geq3$. For the second term, we have
\begin{align}
&\sum_{(\mu_1,\mu_2)\in J_1}\frac{\LL(|\mu_1|+|\mu_2|)^m}{|\mu_1|^{k_1}|\mu_2|^{k_2}}|f_j(\mu_1,\lambda)-f_j(\mu_2,\lambda)|\\
&\leq\summ{|\mu_1|>\frac{|\lambda|}{2},\mu_2\in\latzero}\frac{\LL(|\mu_1|+|\mu_2|)^m}{|\mu_1|^{k_1}|\mu_2|^{k_2}}\left(|f_j(\mu_1,\lambda)|+|f_j(\mu_2,\lambda)|\right)\\
&\ll_m\Bigg(\summ{|\mu_1|>\frac{|\lambda|}{2},\mu_1\neq\lambda}\frac{\LL(|\mu_1|)^m}{|\mu_1|^{k_1}|\mu_1-\lambda|^{2j}}\Bigg)\left(\summ{\mu_2\neq0}\frac{\LL(|\mu_2|)^m}{|\mu_2|^{k_2}}\right)\\
&\quad+\left(\summ{|\mu_1|>\frac{|\lambda|}{2}}\frac{\LL(|\mu_1|)^m}{|\mu_1|^{k_1}}\right)\left(\summ{\mu_2\neq0,\lambda}\frac{\LL(|\mu_2|)^m}{|\mu_2|^{k_2}|\mu_2-\lambda|^{2j}}\right)\,.
\end{align}
By \cref{lem:shifted single eis}, the first product is bounded by the following
\begin{align}
\frac{2}{|\lambda|}\Bigg(\summ{|\mu_1|>\frac{|\lambda|}{2},\mu_1\neq\lambda}\frac{\LL(|\mu_1|)^m}{|\mu_1|^{k_1-1}|\mu_1-\lambda|^{2j}}\Bigg)\left(\summ{\mu_2\neq0}\frac{\LL(|\mu_2|)^m}{|\mu_2|^{k_2}}\right)\ll_{C,m}\frac{1}{|\lambda|}\cdot\frac{\LL(|\lambda|)^{m+1}}{|\lambda|^2}\,.
\end{align}
By \cref{lem:single eis2} and \cref{lem:shifted single eis}, the second product is $\ll_{C,m}\frac{\LL(|\lambda|)^{m}}{|\lambda|^{k_1-2}}\cdot\frac{\LL(|\lambda|)^{m+1}}{|\lambda|^{2}}\ll_{C,m}\frac{\LL(|\lambda|)^{2m+1}}{|\lambda|^{3}}$. For the third term, we can reduce the proof to the case of the second term since we have $J_2\subset\left\{(\mu_1,\mu_2)\relmid|\mu_2|>\frac{|\lambda|}{2}\right\}$. Therefore, we have \eqref{eq:key estim double} and \eqref{eq:key estim single} follows the same way by replacing $\mu_2$ with $0$.
\end{proof}

\begin{lem}\label{lem:twotwo}
Let $\nu\in\ZZ_{>0}\cup\{\infty\}$ and $r\geq0$. For $\tau\in C$ and lattice points $\mu_1,\mu_2\in\Lambda_\tau$ with $\mu_1\succ_\nu\mu_2$, we have
\begin{align}
\left|G^{(\nu;\mu_1,\mu_2)}_{\{2\}^r}(\tau)\right|&\ll_{C,r}\LL(|\mu_1|+|\mu_2|)^r\,.
\end{align}
\end{lem}

\begin{proof}
By definition of $G^{(\nu;\mu_1,\mu_2)}(\tau)$, we have
\begin{align}
\left|G^{(\nu;\mu_1,\mu_2)}_{\{2\}^r}(\tau)\right|&\leq\summ{\lambda_1,\dots,\lambda_r\in\latzero\\\mu_1\succ_\nu\lambda_1\succ_\nu\cdots\succ_\nu\lambda_r\succ_\nu\mu_2}\frac{1}{|\lambda_1|^2\cdots|\lambda_r|^2}\leq\Bigg(\summ{\lambda\in\latzero\\\mu_1\succ_\nu\lambda\succ_\nu\mu_2}\frac{1}{|\lambda|^2}\Bigg)^r\,.
\end{align}
By \cref{lem:single eis1} with $m=0$, we have the conclusion.
\end{proof}

\begin{lem}\label{lem:diff twotwo estim}
Let $N,r\in\ZZ_{>0}$. For $\tau\in C$ and $\mu_1,\mu_2\in\Lambda_\tau$ with $\mu_1\ordN\mu_2$ and $\mu_1\succ\mu_2$, we have
\begin{align}\label{eq:diff twotwo estim}
\left|G^{(N;\mu_1,\mu_2)}_{\{2\}^r}(\tau)-G^{(\infty;\mu_1,\mu_2)}_{\{2\}^r}(\tau)\right|\ll_{C,r}\LL(|\mu_1|+|\mu_2|)^{r-1}\sum_{j=1}^r\left|G^{(N;\mu_1,\mu_2)}_{2j}(\tau)-G^{(\infty;\mu_1,\mu_2)}_{2j}(\tau)\right|\,.
\end{align}
\end{lem}

\begin{proof}
In this proof, we omit $\mu_1,\mu_2,\tau$. By the standard fact of the harmonic product, we have
\begin{align}
rG^{(\nu)}_{\{2\}^r}=\sum_{j=1}^r(-1)^{j-1}G^{(\nu)}_{\{2\}^{r-j}}G^{(\nu)}_{2j}\,.
\end{align}
Therefore, we have
\begin{align}
G^{(N)}_{\{2\}^r}-G^{(\infty)}_{\{2\}^r}
&=\frac{1}{r}\sum_{j=1}^r(-1)^{j-1}\bigg\{\left(G^{(N)}_{\{2\}^{r-j}}-G^{(\infty)}_{\{2\}^{r-j}}\right)G^{(N)}_{2j}+G^{(\infty)}_{\{2\}^{r-j}}\left(G^{(N)}_{2j}-G^{(\infty)}_{2j}\right)\bigg\}\,.\label{eq:diff twotwo}
\end{align}
Now, we prove \eqref{eq:diff twotwo estim} by induction on $r\geq1$. The case $r=1$ is clear. When $r>1$, the term with $j=r$ in the first sum vanishes, since $G^{(N)}_{\{2\}^0}=G^{(\infty)}_{\{2\}^0}=1$ by the assumptions on $\mu_1,\mu_2$. For $j<r$, by inductive hypothesis and \cref{lem:single eis1}, we have
\begin{align}\label{eq:diff1}
\left|G^{(N)}_{\{2\}^{r-j}}-G^{(\infty)}_{\{2\}^{r-j}}\right|\left|G^{(N)}_{2j}\right|\ll_{C,r,j}
\begin{cases}
\LL(|\mu_1|+|\mu_2|)^{r-1}\sum_{i=1}^{r-1}\left|G^{(N)}_{2i}-G^{(\infty)}_{2i}\right|&j=1,\\
\LL(|\mu_1|+|\mu_2|)^{r-j-1}\sum_{i=1}^{r-j}\left|G^{(N)}_{2i}-G^{(\infty)}_{2i}\right|&j=2,\dots,r-1.
\end{cases}
\end{align}
Furthermore, by \cref{lem:twotwo}, we have
\begin{align}\label{eq:diff2}
\sum_{j=1}^r\left|G^{(\infty)}_{\{2\}^{r-j}}\right|\left|G^{(N)}_{2j}-G^{(\infty)}_{2j}\right|\ll_{C,r}\LL(|\mu_1|+|\mu_2|)^{r-1}\sum_{j=1}^r\left|G^{(N)}_{2j}-G^{(\infty)}_{2j}\right|\,.
\end{align}
Combining \eqref{eq:diff twotwo}, \eqref{eq:diff1} and \eqref{eq:diff2}, we obtain \eqref{eq:diff twotwo estim}.
\end{proof}

\begin{lem}\label{lem:key estim2}
Let $k,k_1,k_2\in\ZZ_{\geq3}$ and $r\geq0$. There exists a real number $M\geq0$ such that for $\tau\in C$ and $N\in\ZZ_{>0}$, we have
\begin{align}
&\summ{\mu_1,\mu_2\in\latzero}\frac{\LL(|\mu_1|+|\mu_2|)^m}{|\mu_1|^{k_1}|\mu_2|^{k_2}}\left|G^{(N;\mu_1,\mu_2)}_{\{2\}^r}(\tau)-G^{(\infty;\mu_1,\mu_2)}_{\{2\}^r}(\tau)\right|\ll_{C,k_1,k_2,r,m}\frac{\LL(N)^M}{N^2},\label{eq:main estim1}\\
&\summ{\mu\in\latzero}\frac{\LL(|\mu|)^m}{|\mu|^k}\left|G^{(N;\mu,0)}_{\{2\}^r}(\tau)-G^{(\infty;\mu,0)}_{\{2\}^r}(\tau)\right|\ll_{C,k,r,m}\frac{\LL(N)^M}{N^2}\,.\label{eq:main estim2}
\end{align}
\end{lem}

\begin{proof}
First, we prove \eqref{eq:main estim1}. For $r=0$ the difference is $\ind_{\mu_1\ordN\mu_2}-\ind_{\mu_1\succ\mu_2}$, which is $\pm1$ exactly when $\chi_N(\mu_1-\mu_2)\neq0$, and the sum is then bounded by \cref{lem:chiN} together with \cref{lem:shifted single eis} as in the $K_1$ part below. For $r\geq1$, we divide the sum in \eqref{eq:main estim1} into two terms as follows:
\begin{align}
K_0&\coloneqq\{(\mu_1,\mu_2)\in(\latzero)^2\mid\mu_1\ordN\mu_2,\mu_1\succ\mu_2\},\\
K_1&\coloneqq\{(\mu_1,\mu_2)\in(\latzero)^2\mid\mu_1\ordN\mu_2\succ\mu_1\}\cup\{(\mu_1,\mu_2)\in(\latzero)^2\mid\mu_2\ordN\mu_1\succ\mu_2\}\,.
\end{align}
For all other pairs $(\mu_1,\mu_2)$, both $G^{(N;\mu_1,\mu_2)}_{\{2\}^r}(\tau)$ and $G^{(\infty;\mu_1,\mu_2)}_{\{2\}^r}(\tau)$ vanish. By \cref{lem:single diff} and \cref{lem:diff twotwo estim}, we have
\begin{align}
&\summ{(\mu_1,\mu_2)\in K_0}\frac{\LL(|\mu_1|+|\mu_2|)^m}{|\mu_1|^{k_1}|\mu_2|^{k_2}}\left|G^{(N;\mu_1,\mu_2)}_{\{2\}^r}(\tau)-G^{(\infty;\mu_1,\mu_2)}_{\{2\}^r}(\tau)\right|\\
&\ll_{C,r}\summ{(\mu_1,\mu_2)\in K_0}\frac{\LL(|\mu_1|+|\mu_2|)^{m+r-1}}{|\mu_1|^{k_1}|\mu_2|^{k_2}}\sum_{j=1}^r\summ{\lambda\in\latzero}|\chi_N(\lambda)||f_j(\mu_1,\lambda)-f_j(\mu_2,\lambda)|\,.
\end{align}
Using \cref{lem:key estim1} and \cref{lem:chiN}, we have
\begin{align}
&\summ{(\mu_1,\mu_2)\in K_0}\frac{\LL(|\mu_1|+|\mu_2|)^{m+r-1}}{|\mu_1|^{k_1}|\mu_2|^{k_2}}\sum_{j=1}^r\summ{\lambda\in\latzero}|\chi_N(\lambda)||f_j(\mu_1,\lambda)-f_j(\mu_2,\lambda)|\\
&\ll_{C,r,m}\summ{\lambda\in\latzero}|\chi_N(\lambda)|\frac{\LL(|\lambda|)^{2m+2r-1}}{|\lambda|^3}\\
&\ll_{C,r,m}\frac{\LL(N)^{2m+2r-1}}{N^2}\,.
\end{align}
Now, we give an estimate for the series running over $K_1$. Notice that, for $(\mu_1,\mu_2)\in K_1$, one of $G^{(N;\mu_1,\mu_2)}_{\{2\}^r}(\tau)$ and $G^{(\infty;\mu_1,\mu_2)}_{\{2\}^r}(\tau)$ vanishes. From this fact and \cref{lem:twotwo}, we have
\begin{align}
&\summ{(\mu_1,\mu_2)\in K_1}\frac{\LL(|\mu_1|+|\mu_2|)^m}{|\mu_1|^{k_1}|\mu_2|^{k_2}}\left|G^{(N;\mu_1,\mu_2)}_{\{2\}^r}(\tau)-G^{(\infty;\mu_1,\mu_2)}_{\{2\}^r}(\tau)\right|\\
&=\summ{\mu_1\ordN\mu_2\\\mu_1\prec\mu_2}\frac{\LL(|\mu_1|+|\mu_2|)^m}{|\mu_1|^{k_1}|\mu_2|^{k_2}}\left|G^{(N;\mu_1,\mu_2)}_{\{2\}^r}(\tau)\right|
+\summ{\mu_1\prec_N\mu_2\\\mu_1\succ\mu_2}\frac{\LL(|\mu_1|+|\mu_2|)^m}{|\mu_1|^{k_1}|\mu_2|^{k_2}}\left|G^{(\infty;\mu_1,\mu_2)}_{\{2\}^r}(\tau)\right|\\
&\ll_{C,r}\summ{(\mu_1,\mu_2)\in K_1}\frac{\LL(|\mu_1|+|\mu_2|)^{m+r}}{|\mu_1|^{k_1}|\mu_2|^{k_2}}\,.
\end{align}
By $|\mu_1|+|\mu_2|\leq|\mu_1-\mu_2|+2|\mu_2|$, it holds
\begin{align}
\LL(|\mu_1|+|\mu_2|)^{m+r}\ll\LL(|\mu_1-\mu_2|)^{m+r}\LL(|\mu_2|)^{m+r}\,.
\end{align}
Thus, we have
\begin{align}
\summ{(\mu_1,\mu_2)\in K_1}\frac{\LL(|\mu_1|+|\mu_2|)^{m+r}}{|\mu_1|^{k_1}|\mu_2|^{k_2}}\ll\summ{\lambda\in\latzero\\\chi_N(\lambda)\neq0}\LL(|\lambda|)^{m+r}\summ{\mu_2\in\Lambda_\tau\setminus\{0,-\lambda\}}\frac{\LL(|\mu_2|)^{m+r}}{|\mu_2+\lambda|^{k_1}|\mu_2|^{k_2}}\,.
\end{align}
Since $|\mu_2|\geq\frac{|\lambda|}{2}$ or $|\mu_2+\lambda|\geq\frac{|\lambda|}{2}$ holds for any $\mu_2\in\Lambda_\tau\setminus\{0,-\lambda\}$, we have
\begin{align}
\frac{1}{|\mu_2+\lambda|^{k_1}|\mu_2|^{k_2}}\leq\frac{2}{|\lambda|}\left(\frac{1}{|\mu_2+\lambda|^{k_1-1}|\mu_2|^{k_2}}+\frac{1}{|\mu_2+\lambda|^{k_1}|\mu_2|^{k_2-1}}\right)\,,
\end{align}
and since $k_1-1,k_2-1\geq2$, \cref{lem:shifted single eis} gives
\begin{align}
\summ{\mu_2\in\Lambda_\tau\setminus\{0,-\lambda\}}\frac{\LL(|\mu_2|)^{m+r}}{|\mu_2+\lambda|^{k_1}|\mu_2|^{k_2}}\ll_{C,m,r}\frac{\LL(|\lambda|)^{m+r+1}}{|\lambda|^3}\,.
\end{align}
Together with \cref{lem:chiN}, we have
\begin{align}
\summ{\lambda\in\latzero\\\chi_N(\lambda)\neq0}\LL(|\lambda|)^{m+r}\summ{\mu_2\in\Lambda_\tau\setminus\{0,-\lambda\}}\frac{\LL(|\mu_2|)^{m+r}}{|\mu_2+\lambda|^{k_1}|\mu_2|^{k_2}}&\ll_{C,k_1,k_2,m,r}\summ{\lambda\in\latzero\\\chi_N(\lambda)\neq0}\LL(|\lambda|)^{m+r}\frac{\LL(|\lambda|)^{m+r+1}}{|\lambda|^3}\\
&\ll_{C,m,r}\frac{\LL(N)^{2(m+r)+1}}{N^2}\,.
\end{align}
Therefore we have \eqref{eq:main estim1}. \eqref{eq:main estim2} follows the same two-way split with $\mu_2$ replaced by $0$.
\end{proof}

\subsection{Proof of \texorpdfstring{\cref{thm:first-order}}{the main estimate}}
We write $\kk=(k_1,\dots,k_r)=(a_1,\{2\}^{b_1},\dots,a_s,\{2\}^{b_s})$ with $s\geq1$, $a_i\geq3$ and $b_i\geq0$, where $\{2\}^b$ denotes $b$ consecutive entries $2$. This is possible and unique since $k_1\geq3$. For any $N\in\ZZ_{>0}$, the following is just the definition of $G^{(N)}_\kk$ with the sum split at the entries $a_i$, the $\mu_i$ being the lattice points at these entries and $\mu_{s+1}=0$.
\begin{align}
G^{(N)}_\kk(\tau)=\summ{\mu_1,\dots,\mu_s\in\Lambda_\tau\setminus\{0\}\\\mu_{s+1}=0}\prod_{i=1}^s\frac{G^{(N;\mu_i,\mu_{i+1})}_{\{2\}^{b_i}}(\tau)}{\mu_i^{a_i}}\,.
\end{align}
Thus, we have
\begin{align}
G^{(N)}_\kk(\tau)-G_\kk(\tau)=\summ{\mu_1,\dots,\mu_s\in\Lambda_\tau\setminus\{0\}}\frac{\prod_{i=1}^sG^{(N;\mu_i,\mu_{i+1})}_{\{2\}^{b_i}}(\tau)-\prod_{i=1}^sG^{(\infty;\mu_i,\mu_{i+1})}_{\{2\}^{b_i}}(\tau)}{\prod_{i=1}^s\mu_i^{a_i}}\,.
\end{align}
Since $x_1\cdots x_s-y_1\cdots y_s=\sum_{i=1}^sx_1\cdots x_{i-1}(x_i-y_i)y_{i+1}\cdots y_s$, we can write the difference as follows.
\begin{align}
G^{(N)}_\kk(\tau)-G_\kk(\tau)=\sum_{i=1}^sD_{N,i}\,,
\end{align}
where
\begin{align}
D_{N,i}\coloneqq\summ{\mu_1,\dots,\mu_s\in\Lambda_\tau\setminus\{0\}}&\left(\prod_{j=1}^{i-1}\frac{G^{(N;\mu_j,\mu_{j+1})}_{\{2\}^{b_j}}(\tau)}{\mu_j^{a_j}}\right)
\left(\prod_{j=i+1}^s\frac{G^{(\infty;\mu_j,\mu_{j+1})}_{\{2\}^{b_j}}(\tau)}{\mu_j^{a_j}}\right)\\
&\times\frac{G^{(N;\mu_i,\mu_{i+1})}_{\{2\}^{b_i}}(\tau)-G^{(\infty;\mu_i,\mu_{i+1})}_{\{2\}^{b_i}}(\tau)}{\mu_i^{a_i}}\,.
\end{align}
By \cref{lem:twotwo}, we have
\begin{align}
\left|D_{N,i}\right|\ll_{C,\kk}\summ{\mu_1,\dots,\mu_s\in\Lambda_\tau\setminus\{0\}}\left(\prod_{j=1}^s\frac{\LL(|\mu_j|)^{m_{i,j}}}{|\mu_j|^{a_j}}\right)\left|G^{(N;\mu_i,\mu_{i+1})}_{\{2\}^{b_i}}(\tau)-G^{(\infty;\mu_i,\mu_{i+1})}_{\{2\}^{b_i}}(\tau)\right|\,,
\end{align}
where $m_{i,j}=\ind_{j\neq i+1}b_{j-1}+\ind_{j\neq i}b_j$ with $b_0\coloneqq0$. Since $\textstyle\sum_{\mu_j\in\Lambda_\tau\setminus\{0\}}\frac{\LL(|\mu_j|)^{m_{i,j}}}{|\mu_j|^{a_j}}$ converges uniformly for $\tau\in C$ for $j\neq i,i+1$, we have
\begin{align}
|D_{N,i}|\ll_{C,\kk}\summ{\mu_i,\mu_{i+1}\in\Lambda_\tau\setminus\{0\}}\frac{\LL(|\mu_i|)^{b_{i-1}}\LL(|\mu_{i+1}|)^{b_{i+1}}}{|\mu_i|^{a_i}|\mu_{i+1}|^{a_{i+1}}}\left|G^{(N;\mu_i,\mu_{i+1})}_{\{2\}^{b_i}}(\tau)-G^{(\infty;\mu_i,\mu_{i+1})}_{\{2\}^{b_i}}(\tau)\right|
\end{align}
for $1\leq i\leq s-1$. And for $i=s$, we have
 \begin{align}
 |D_{N,s}|\ll_{C,\kk}\sum_{\mu_s\in\latzero}\frac{\LL(|\mu_s|)^{b_{s-1}}}{|\mu_s|^{a_s}}\left|G^{(N;\mu_s,0)}_{\{2\}^{b_s}}(\tau)-G^{(\infty;\mu_s,0)}_{\{2\}^{b_s}}(\tau)\right|\,.
 \end{align}
By \cref{lem:key estim2}, there exists a real number $M\geq0$ such that for any $i\in\{1,\dots,s\}$
\begin{align}
|D_{N,i}|\ll_{C,\kk}\frac{\LL(N)^M}{N^2}\,.
\end{align}
Therefore, we have $\left|G^{(N)}_\kk(\tau)-G_\kk(\tau)\right|\leq\sum_{i=1}^s|D_{N,i}|\ll_{C,\kk}\frac{\LL(N)^M}{N^2}$.
\qed

\smallskip
\begingroup
\small
\vspace{0.5cm}
\noindent {\bf AI \& computational resource disclosure:} The results, their formulation, and the proof strategy are the authors' own.  ChatGPT~5.6 and Claude Fable~5 were used as research assistants to check computations (via SageMath) and develop intermediate proof steps and the language. The authors verified all statements, proofs, and references and take full responsibility for them.
\par
\endgroup


\begin{thebibliography}{99}

\bibitem[Ba1]{Ba1} H. Bachmann, \textit{Multiple Zeta-Werte und die Verbindung zu Modulformen durch Multiple Eisensteinreihen}, Master thesis, Universit\"at Hamburg, 2012.

\bibitem[BKM]{BKM} H. Bachmann, H. Kanno and T. Maesaka, \textit{Relations and derivatives of multiple Eisenstein series}, preprint, arXiv:2602.08176, 2026.

\bibitem[BIM]{BIM} H. Bachmann and J.-W. van Ittersum, \textit{Formal multiple Eisenstein series and their derivations} (with Appendix by N. Matthes), Adv. Math. \textbf{487} (2026), Paper No. 110739.

\bibitem[BK1]{BK1} H. Bachmann and U. K\"uhn, \textit{The algebra of generating functions for multiple divisor sums and applications to multiple zeta values}, Ramanujan J. \textbf{40} (2016), 605--648.

\bibitem[BK2]{BK2} H. Bachmann and U. K\"uhn, \textit{A dimension conjecture for $q$-analogues of multiple zeta values}, in \textit{Periods in quantum field theory and arithmetic}, Springer Proc. Math. Stat. \textbf{314}, Springer, Cham, 2020, 237--258.

\bibitem[BT]{BT} H. Bachmann and K. Tasaka, \textit{The double shuffle relations for multiple Eisenstein series}, Nagoya Math. J. \textbf{230} (2018), 180--212.

\bibitem[DS]{DS} F.~I. Diamond and J. Shurman, \textit{A first course in modular forms}, Graduate Texts in Mathematics, 228, Springer, New York, 2005.

\bibitem[GKZ]{GKZ} H. Gangl, M. Kaneko and D. Zagier, \textit{Double zeta values and modular forms}, in \textit{Automorphic forms and zeta functions}, World Scientific, Singapore, 2006, 71--106.

\bibitem[HST]{HST} T. Hara, K. Sakugawa and K. Tasaka, \textit{Symmetric multiple Eisenstein series}, preprint, arXiv:2601.13626, 2026.

\bibitem[O]{O} A. Okounkov, \textit{Hilbert schemes and multiple $q$-zeta values}, Funct. Anal. Appl. \textbf{48} (2014), 138--144.

\bibitem[Za]{Za} D. Zagier, \textit{Elliptic modular forms and their applications}, in \textit{The 1-2-3 of modular forms}, Universitext, Springer, Berlin, 2008, 1--103.

\end{thebibliography}
\end{document}